\documentclass[twoside,11pt]{article}

\usepackage{amsthm}
\usepackage{jmlr2e}

\usepackage{amsmath}

\usepackage{xspace}
\usepackage{dsfont}
\usepackage[cal=boondoxo]{mathalfa}

\usepackage{booktabs}

\usepackage{cleveref}

\newtheorem{assumption}{Assumption}

\def\argmin{\mathop{\rm arg \; min}\limits}\def\argmax{\mathop{\rm arg \; max}\limits}

\newcommand{\E}{{\mathbb E}}
\newcommand{\R}{{\mathbb R}}

\newcommand{\N}{{\mathbb N}}

\renewcommand{\P}{{\mathbb P}}
\newcommand{\Id}{{\mathrm{Id}}}
\newcommand{\indic}{{\mathds{1}}}
\newcommand{\bX}{{\mathbf{X}}}
\newcommand{\bY}{{\mathbf{Y}}}
\newcommand{\bZ}{{\mathbf{Z}}}
\newcommand{\bx}{{\mathbf{x}}}
\newcommand{\by}{{\mathbf{y}}}

\newcommand{\JJ}{\ensuremath{\mathcal J}}
\newcommand{\hJJ}{\ensuremath{\widehat{\mathcal J}}}
\newcommand{\tJJ}{\ensuremath{\widetilde{\mathcal J}}}
\newcommand{\LL}{\ensuremath{\mathcal L}}

\newcommand{\TT}{\ensuremath{\mathcal T}}

\newcommand{\Sc}{\ensuremath{\mathcal S}}
\newcommand{\hSc}{\ensuremath{\widehat{\mathcal S}}}
\newcommand{\tSc}{\ensuremath{\widetilde{\mathcal S}}}

\newcommand{\ff}{\ensuremath{\mathcal f}}
\newcommand{\hff}{\ensuremath{\widehat{\mathcal f}}}
\newcommand{\tff}{\ensuremath{\widetilde{\mathcal f}}}
\newcommand{\RR}{\ensuremath{\mathcal R}}
\newcommand{\hphi}{\ensuremath{\widehat{\phi}}}
\newcommand{\tphi}{\ensuremath{\widetilde{\phi}}}
\newcommand{\htau}{\ensuremath{\widehat{\tau}}}
\newcommand{\Ac}{\ensuremath{\mathcal A}}
\newcommand{\Bc}{\ensuremath{\mathcal B}}

\newcommand{\Pc}{\ensuremath{\mathcal P}}

\newcommand{\MM}{\ensuremath{\mathcal M}}
\newcommand{\g}{\ensuremath{\mathcal{g}}}

\def\argmin{\mathop{\rm arg \; min}\limits}\newcommand{\supp}{\mathop{\mathrm{supp}}}
\DeclareMathOperator{\diam}{diam}
\newcommand{\Xsp}{\ensuremath{\mathds{X}}\xspace}
\newcommand{\Ysp}{\ensuremath{\mathds{Y}}\xspace}

\newcommand{\Vsp}{\ensuremath{\mathds{V}}\xspace}
\newcommand{\Msp}{\ensuremath{\mathcal{M}}\xspace}

\newcommand{\Cov}{\mathrm{Cov}}
\DeclareMathOperator*{\esssup}{ess\,sup}

\usepackage{lastpage}
\jmlrheading{25}{2024}{1-\pageref{LastPage}}{9/23; Revised 8/24}{8/24}{23-1149}{Paul Escande}

\ShortHeadings{On the Concentration of the Minimizers of Empirical Risks}{Paul Escande}
\firstpageno{1}

\begin{document}

\title{On the Concentration of the Minimizers of Empirical Risks}

\author{\name Paul Escande \email paul.escande@math.univ-toulouse.fr \\
       \addr Institut de Mathématiques de Toulouse\\
       UMR 5219, Université de Toulouse, CNRS\\
       UPS, F-31062 Toulouse Cedex 9, France}

\editor{Joseph Salmon}

\maketitle

\begin{abstract}
Obtaining guarantees on the convergence of the minimizers of empirical risks to the ones of the true risk is a fundamental matter in statistical learning. 
Instead of deriving guarantees on the usual estimation error, the goal of this paper is to provide concentration inequalities on the distance between the sets of minimizers of the risks for a broad spectrum of estimation problems.
In particular, the risks are defined on metric spaces through probability measures that are also supported on metric spaces. A particular attention will therefore be given to include unbounded spaces and non-convex cost functions that might also be unbounded.
This work identifies a set of high-level assumptions allowing to describe a regime that seems to govern the concentration in many estimation problems, where the empirical minimizers are stable. This stability can then be leveraged to prove parametric concentration rates in probability and in expectation.
The assumptions are verified, and the bounds showcased, on a selection of estimation problems such as barycenters on metric space with positive or negative curvature, subspaces of covariance matrices, regression problems and entropic-Wasserstein barycenters. \end{abstract}


\section{Introduction}

Let $(\Xsp,\rho, \mu)$ be a metric probability space\footnote{A measurable space $\Xsp$ whose Borel $\sigma$-algebra is induced by $\rho$ and is endowed with a probability measure $\mu$.} and let $(\MM,\vartheta)$ be another metric space.
Many estimation problems in statistics or in machine learning can be regarded as solving the following generic optimization problem
\begin{equation} \label{eq:opt_J}
  \begin{aligned}
    \phi_* \in & \argmin_{\phi \in \MM} \JJ(\phi) = \E_{X \sim \mu}[l(\phi,X)] + \RR(\phi),
  \end{aligned}
\end{equation}
where $l : \MM \times \Xsp \to \R$ is a cost function and $\RR : \MM \to \R \cup \{ +\infty \}$ is an optional regularization function promoting the a priori structure of the solution. The function $\JJ : \MM \to \R$ is oftentimes referred to as the risk or the objective function. 
For convenience, the data fidelity term $\ff : \MM \to \R$ will be defined by $\ff(\phi) = \E_{X \sim \mu}[l(\phi,X)]$.

In practice, the measure $\mu$ is rarely known thus hindering the knowledge of $\JJ$. However, samples drawn from $\mu$ are usually accessible.
Let $\bX = (X_i)_{i=1}^n \in \Xsp^n$ be such samples that will further be assumed to be independent. 
The main idea of the Empirical Risk Minimization (ERM) \citep{vapnik1991principles} is to use these samples to build $\hff$, an estimator of $\ff$, as $\hff(\phi) = \frac{1}{n} \sum_{i=1}^n l(\phi,X_i)$ and to solve 
\begin{equation}  \label{eq:opt_hatJ}
\widehat{\phi} \in \argmin_{\phi \in \MM} \hJJ(\phi) = \hff(\phi) + \RR(\phi),
\end{equation}
in place of \eqref{eq:opt_J}.

Understanding when the estimator $\hphi$ is legit is a fundamental matter in statistical learning. 
This work will be particularly interested in proving high probability bounds of kind
\begin{equation} \label{eq:PAC_framework}
  \P\left( \vartheta(\argmin \hJJ ; \argmin \JJ) < g(\delta,n) \right) \geq 1-\delta, \quad \text{for } \delta \in [0,1],
\end{equation}
where $\vartheta(A ; B)$ will be a suitable distance between sets of $\Msp$ and $g : [0,1] \times \N \to \R_+$ going to zero as $n$ goes to infinity for all fixed $\delta$.
This question is very similar to the Probably Approximately Correct (PAC) framework \citep{valiant1984theory} with the important distinction that it proves bound on the estimation error $\JJ(\hphi) - \inf \JJ$, also known as the excess risk, instead of a metric on the minimizers.
More specifically, this work attempts to identify a set of assumptions from which optimal concentration rates of kind \eqref{eq:PAC_framework} can be derived for a broad class of estimation problems. In essence, it is proposed here to understand the underlying mechanisms responsible for the concentration guarantees of many estimation problems.

These fundamental questions have a long history and have led to a remarkably rich theory.
In order to precisely contextualize this work, it seemed essential to dedicate a section to present its history and some important works (see Section \ref{sec:review_literature}).
Prior to this literature review, some occurrences of the estimation problems \eqref{eq:opt_hatJ} solved in place of \eqref{eq:opt_J} are stated with the aim of describing the taxonomy of problems encountered in applications.

\subsection{Applications} \label{sec:intro_applications}

Some occurrences of \eqref{eq:opt_hatJ} are briefly discussed in this section to describe the taxonomy of problems encountered in applications. 
They will be studied in detail in Section \ref{sec:applications}.

\begin{description}
\item[Barycenters:] Let $(\Msp,\vartheta) = (\Xsp,\rho)$ be a metric space, $l(\phi,x) = \rho(x,y)^2$ and $\mathcal{R}$ be identically zero. The problems \eqref{eq:opt_J} and \eqref{eq:opt_hatJ} become 
\begin{equation} \label{eq:frechet_means}
  \phi_* \in \argmin_{\phi \in \Xsp} \int_{\Xsp} \rho^2(\phi,x) d\mu(x), \quad \widehat{\phi} \in \argmin_{\phi \in \Xsp} \frac{1}{n} \sum_{i=1}^n \rho^2(\phi,x_i),
\end{equation}
which are the Fréchet means of $\mu$ and $\mu_n = \frac{1}{n} \sum_{i=1}^n\delta_{x_i}$ respectively \citep{frechet1948elements}.
The properties of the risks $\JJ$ and $\hJJ$ depend on the convexity and the curvature of the space $\Msp$. The risks are known to be strongly convex when the space $\Msp$ are of non-positive curvature in the sense of Alexandrov \citep[see][]{sturm2003probability}. Otherwise, the risks might not be convex since the distance function might exhibit some concavity properties \citep[see][]{ambrosio2008gradient}.

\item[Subspace estimation:]
Let $(\Xsp,\rho) = (\R^d, \|\cdot\|_2)$, $\Msp = \mathbb{S}^{d-1}$ endowed with the great circle distance and $\mathcal{R}$ be identically zero. 
The problem of estimating the leading direction of variability of $\Cov(\mu) = \E\left[X X^T \right]$, the covariance of a centered measure $\mu$, amounts to solving
\begin{equation*}
  \widehat{\phi} \in \argmax_{\phi \in \mathbb{S}^{d-1}} \langle \Cov(\mu_n) \phi,\phi \rangle, \quad \text{in place of } \quad \phi_* \in \argmax_{\phi \in \mathbb{S}^{d-1}} \langle \Cov(\mu) \phi,\phi \rangle,
\end{equation*}
which can be seen as an instance of \eqref{eq:opt_J} and \eqref{eq:opt_hatJ} with $l(\phi,x) = - \langle \phi, x \rangle^2$, and $\mathcal{R}$ being identically zero.
These problems are non-convex.

\item[Supervised learning -- LASSO:]

Let $\Ysp = \R^d$ and $\Vsp = \R$ be respectively a space of inputs and outputs.
Let $\Xsp = \Ysp \times \Vsp$ and $\Msp = \R^m$ endowed with the Euclidean distance. Let $\theta : \R^d \to \R^m$ be feature vectors, $l : \Xsp \to \R_+$ defined for all $(y,v) \in \Xsp$ as $l(\phi,(y,v)) = \frac{1}{2} (\langle \phi, \theta(y)\rangle  - v)^2$ and $\RR = \lambda \| \cdot \|_1$ for $\lambda > 0$.
The problems \eqref{eq:opt_J} and \eqref{eq:opt_hatJ} instantiate to
\begin{equation*}
  \phi_* \in \argmin_{\phi \in \R^m} \frac{1}{2} \int_{\Xsp} (\langle \phi, \theta(y) \rangle - v)^2 d\mu(y,v) + \lambda \| \phi \|_1,
\end{equation*}
and 
\begin{equation*}
  \widehat{\phi} \in \argmin_{\phi \in \R^m} \frac{1}{2n} \| V - \Theta \phi\|_2^2 + \lambda \| \phi \|_1,
\end{equation*}
with $V \in \R^n$ being the vector of outputs $V = (v_i)_{i=1}^n$ and $\Theta \in \R^{n \times m}$ is the design matrix with the vectors $(\theta(y_i)^T)_{i=1}^n$ as rows.
The risks $\JJ$ and $\hJJ$ are convex, but they may fail to be strongly convex if the matrix $\Cov(\theta \sharp \mu)$ (respectively with $\mu_n$) has a non-empty kernel. In this case, the risks posses a quadratic growth around the minimizers in the orthogonal of the kernel and a linear growth in the kernel driven by the regularization.

\end{description}

These three examples are simple yet important to understand the assumptions needed to describe the concentration phenomenon of most of optimization problems. In particular, the loss $l$ or the objectives $\JJ$ and $\hJJ$ cannot be assumed to be strongly convex or even convex. Similarly, the spaces $\Msp$ and $\Xsp$ at stake might be unbounded.

\subsection{Problem History and Related Works} \label{sec:review_literature}

Solving \eqref{eq:opt_hatJ} in place \eqref{eq:opt_J} is related to the principle of ERM or of Structural Risk Minimization (SRM). The study of theoretical guarantees for the estimator $\hphi$, in terms of bounds of kind \eqref{eq:PAC_framework} for the estimation error, has a long history which can possibly be tracked back to the 1970's with the works of \citet{vapnik1971uniform,vapnik1974theory}.
Two main classes of analyses can apparently be identified: the techniques using uniform convergence and the ones from algorithm stability. 
They are described in the following two paragraphs and will be linked to the question at stake in the third paragraph, that is deriving bounds of type \eqref{eq:PAC_framework} for the metric distance between the minimizers.
While being insufficient to fully explain the optimal error bounds that are observed in \eqref{eq:PAC_framework}, their presentation was found important to articulate the contributions of this work.
Moreover, the tools used to achieve the error bounds \eqref{eq:PAC_framework} share some similarities with these previous works.
Finally, it should be mentioned that the connection between uniform convergence and stability have been investigated in various works \citep[such as][and reference therein]{shalev2010learnability}.

It is assumed in this section that the minimizers of $\JJ$ and $\hJJ$ exist and are unique. The minimal value of $\JJ$ will be referred to as $\JJ_*$ with similar notation for other functions.

\subsubsection{Uniform Convergence}
The techniques in the uniform convergence class can be tracked back to the 1970's with the works of \cite{vapnik1971uniform,vapnik1974theory,vapnik1991principles} and the PAC framework proposed in the 1980's by \cite{valiant1984theory}.
Since then, their seminal works have generated a large body of literature and a mature theory. A complete introduction of these notions can be found in good textbooks \citep[such as][to name a few]{devroye2013probabilistic,shalev2014understanding,wainwright2019high,vershynin2018high,bach2021learning}.

The starting point of these works is the decomposition of the estimation error
\begin{equation} \label{eq:estimation_error_decomposition}
  \begin{aligned}
    \JJ(\hphi) - \JJ_* &\leq \JJ(\hphi) - \JJ(\phi_*) + \hJJ(\phi_*) - \hJJ(\hphi) && (\text{since $\hphi \in \argmin \hJJ$}) \\
    & = \ff(\hphi) - \ff(\phi_*) + \hff(\phi_*) - \hff(\hphi) && (\text{since the terms in } \RR \text{ cancel out}) \\
    & \leq 2 \sup_{\phi \in \Msp} \Delta(\phi) &&
    \end{aligned}
\end{equation}
where $\Delta(\phi) = |\ff(\phi) - \hff(\phi)|$. The uniform bound is taken to remove the dependency between $\hphi$ and $\hff$ that was preventing the $(l(\hphi,X_i))_{i=1}^n$ to be independent.
This allows to use standard concentration tools such as the McDiarmid's inequality \citep{mcdiarmid1998concentration} (assuming for simplicity that the cost function $0 \leq l \leq M$ for some $M > 0$) leading to the following generic bound on the estimation error
\begin{equation} \label{eq:generic_bound_uniform_cv}
\JJ(\hphi) - \JJ_* \leq 2\E\left[ \sup_{\phi \in \Msp} \Delta(\phi) \right] + M \sqrt{ \frac{2}{n} \log\left(\frac{2}{\delta} \right)},
\end{equation} 
with high probability at least $1 - \delta$. 
The downside of this uniform bound is to control $\E\left[ \sup_{\phi \in \Msp} \Delta(\phi) \right]$ (that is, the expectation of the supremum of a random process) which can be quite large and hindering any exploitation of a potential structure of the minimizers.
Yet, this control can be achieved in mainly two ways.

First, via symmetrization, the latter can further be bounded by the Rademacher complexity or the Gaussian complexity of the function class $\{(l(\phi,x_i))_{i=1}^n \, | \, \phi \in \Msp\}$ \citep{bartlett2002rademacher}. Note that when the values of $l$ are discrete, the Rademacher complexity can be linked to the VC dimension of the class of functions \citep{bousquet2003introduction}. 
The Rademacher complexities of the typical function classes encountered in estimation behave as $c / \sqrt{n}$, where $c > 0$ is a constant depending on the problem \citep{wainwright2019high,shalev2014understanding}. Note that the use of Rademacher complexities can be improved by considering localized versions \citep{bartlett2005local,koltchinskii2006local}. 

Second, when the cost $l$ exhibits some smoothness properties (for example Lipschitz) with respect to the variable $\phi$, the quantity $\E\left[ \sup_{\phi \in \Msp} \Delta(\phi) \right]$ can alternatively be bounded using chaining techniques such as Dudley's chaining or the Talagrand's generic chaining \citep{talagrand2014upper,vershynin2018high}. This leads to an upper-bound dependent on a measure of complexity of $\Msp$, such as the entropy of $\Msp$, linked to the covering numbers of $\Msp$, or the Talagrand's $\gamma_2$ functional. These complexities typically behave as $n^{-1/(2s)}$ with $s = 1$ for simple spaces such as Euclidean spaces, or doubling spaces, and with $s > 1$ for more complex spaces \citep[see][]{ahidar2020convergence,schotz2019convergence}. In the case $s > 1$, the rate is customarily said to be \emph{cursed by the dimension}.

To summarize, the bounds on the estimation error obtained with these techniques of uniform convergence decrease as $n^{-1/2}$. Without further assumption on the cost function $l$, this rate is optimal \citep[see][]{mendelson2008lower}.

\subsubsection{Algorithm Stability}

Another body of literature studies the generalization of learning algorithms through stability. 
These ideas can be tracked back to the 1970's as well with \cite{rogers1978finite} and later popularized in the founding article of \citet{bousquet2002stability}.
The notions involved in these works will be introduced here for the case of the empirical minimizers, although they can be applied to the broader setting of learning algorithms.\footnote{More precisely, a learning algorithm is a function $\Ac : \Xsp^n \to \Msp$ that maps the given data $\bX = (X_i)_{i=1}^n$ to an element of $\Msp$ which can be the minimizer $\hphi$ of the empirical risk $\hJJ$ or an approximation of it obtained using any algorithm (for example gradient descent).}

For convenient notation, let $\bX = (X_1,\ldots,X_n) \in \Xsp^n$ be the random vector containing the $n$ independent sample points, and let $Y \in \Xsp$ be an independent copy of $X_1$. The random vector $\bX'$ is obtained by swapping one element of $\bX$ with $Y$ so that $\bX' = (Y,X_2,\ldots,X_n)$. Furthermore, the empirical risk associated to $\bX'$ will be denoted by $\tJJ$ and its minimizer by $\tphi$.
It is important to remark at this point that the risk $\hJJ$ and therefore its minimizers are invariant under reordering of the components of the vector $\bX$. The perturbation of $\bX$ by swapping any element can therefore be thought as exchanging the first component.

The starting point of these works is a different decomposition of the estimation error
\begin{equation} \label{eq:estimation_error_decomposition_stability}
  \begin{aligned}
  \E_\bX\left[ \JJ(\hphi) - \JJ_* \right] &\leq \E_\bX\left[\JJ(\hphi) - \hJJ(\hphi)\right] && (\text{using Equation \ref{eq:estimation_error_decomposition}}) \\
  & = \E_\bX\left[\E_Y[l(\hphi,Y)] - l(\hphi,X_1)\right] && (\text{since } \E[l(\hphi,X_i)] = \E[l(\hphi,X_j)] \text{ for all } i,j) \\
  & = \E_{\bX,Y}\left[l(\tphi,X_1) - l(\hphi,X_1)\right], && (\text{with } \tphi = \argmin \tJJ)
  \end{aligned}
\end{equation}
because the roles of $Y$ and $X_1$ can be switched as being independent copies.
This decomposition emphasizes the link between the estimation error and the \emph{stability} of the empirical minimizer, that is if it does not vary much when a sample of the data $\bX$ is perturbed or removed. Various notions of stability have been developed such as the hypothesis stability \citep{rogers1978finite,kearns1997algorithmic} or the uniform one \citep{bousquet2002stability}. 
The latter is probably the most popular one since it can be used to obtain Gaussian concentration bounds.
The empirical minimizer is said to be $\epsilon$-uniform stable if there exists $\epsilon > 0$ such that $|l(\hphi,x) - l(\tphi,x)| \leq \epsilon$, for all $x \in \Xsp$ and all $(\bX, Y) \in \Xsp^{n+1}$. This stability can be leveraged to directly bound the expected estimation error as $\E_\bX\left[ \JJ(\hphi) - \JJ_* \right] \leq \epsilon$, and obtain \cite[see][Theorem 12]{bousquet2002stability} that 
\begin{equation} \label{eq:thm_uniform_stability}
\JJ(\hphi) - \JJ_* \leq \epsilon + (4n\epsilon + M) \sqrt{\frac{1}{n} \log\left(\frac{1}{\delta}\right)},
\end{equation}
with high probability $1 - \delta$ whenever $0 \leq l \leq M$. 
This bound is informative as long as the stability $\epsilon \leq c n^{-1}$ for some constant $c>0$.
In contrast of \eqref{eq:generic_bound_uniform_cv}, the stability of the empirical minimizers avoids the uniform bound on $\Msp$, which is desirable since the empirical solutions may exhibit a very strong structure. Importantly, the bound \eqref{eq:thm_uniform_stability} cannot be cursed by the dimension since it cannot depend exponentially on the complexity of the space $\Msp$.

Note that \eqref{eq:thm_uniform_stability} is also obtained using the McDiarmid's inequality and requires the boundedness of the cost $l$. This might narrow the range of application of these results.
This assumption can however be dropped at the expense of a stronger notion of stability and a bounded sub-Gaussian or sub-exponential diameter of $\Xsp$ \citep{kontorovich2014concentration,maurer2021concentration}. 

Recently, a series of papers \citep{feldman2018generalization,feldman2019high,bousquet2020sharper} improved \eqref{eq:thm_uniform_stability} to
\begin{equation} \label{eq:thm_improved_uniform_stability}
\JJ(\hphi) - \JJ_* \leq \epsilon \log(n) \log\left(\frac{1}{\delta}\right) + M \sqrt{\frac{1}{n} \log\left(\frac{1}{\delta}\right)},
\end{equation}
holding with high probability $1-\delta$. This was considered as a breakthrough in the algorithm stability community as the bound becomes informative as long as $\epsilon \leq c n^{-1/2}$ which extends the validity of the result to a broader variety of applications.

In addition of theses refinements, it is relevant for this work to mention that the estimation error can still be bounded when the stability of the empirical minimizers happens only on a subset of large measure \citep{kutin2012almost,rakhlin2005stability}.

\subsubsection{Obtaining Concentration Bounds on $\vartheta$}

Despite a wide interest in bounding the estimation error, little is however known regarding the concentration of the minimizers in the sense of \eqref{eq:PAC_framework}, involving a distance on the minimizers.

As a preliminary, it can be mentioned that when the minimizers of $\JJ$ and $\hJJ$ have explicit expressions with a clear dependency on the measure $\mu$ and the sample $(X_i)_{i=1}^n$, their structure, if any, could be leveraged to directly use concentration tools. This could happen for example when $\JJ$ and $\hJJ$ are quadratic, which is however restrictive for modern machine learning applications.

The first observation towards providing an indirect analysis, is that when an error bound of kind 
\begin{equation}\label{eq:intro_growth}
\tau \vartheta(\hphi,\phi_*)^\beta \leq \JJ(\hphi) - \JJ_*,
\end{equation}
 is available for some $\beta > 0$ and $\tau > 0$, the concentration of $\vartheta(\hphi,\phi_*)$ can be derived using the results of the two previous paragraphs. 
A direct combination of \eqref{eq:intro_growth} and \eqref{eq:thm_improved_uniform_stability} in the example of barycenters in a compact and convex $\Omega \subset \R^d$ would give 
\begin{equation} \label{eq:thm_improved_uniform_stability_barycenter}
  \| \hphi - \phi_*\|_2^2 = \JJ(\hphi) - \JJ_* \leq \epsilon \log(n) \log\left(\frac{1}{\delta}\right) + \diam(\Omega)^2 \sqrt{\frac{1}{n} \log\left(\frac{1}{\delta}\right)},
\end{equation}
with high probability $1 - \delta$.
Assuming for example that the random variables $X \sim \mu$ are sub-Gaussian with variance proxy $\|X\|_{\psi_2}$, see Preliminaries section or \citep{vershynin2018high}, their concentration property would instead lead to
\begin{equation} \label{eq:optimal_rate_barycenter}
\| \hphi - \phi_*\|_2^2 \leq c d \| X \|_{\psi_2}^2 \frac{1}{n} \log\left(\frac{2}{\delta}\right),
\end{equation}
with high probability $1 - \delta$ and some universal constant $c > 0$ \citep{jin2019short}. The rate in $n^{-1}$ in \eqref{eq:optimal_rate_barycenter} is called \emph{parametric}.
Even if the stability of the empirical minimizer can be shown to be $\epsilon = O(n^{-1})$, the sampling error in $\sqrt{n^{-1} \log(\delta^{-1})}$ would slow down the rate nevertheless. 
The same slow rate is obtained with the uniform convergence \eqref{eq:generic_bound_uniform_cv}.

This gap is known in the literature \citep[see][]{shalev2009stochastic,shalev2010learnability,klochkov2021stability} and is due to the fact that uniform convergence and stability techniques do not incorporate any structure of the problem.
In particular, the behavior of $\JJ$ around its minimizers can be leveraged, such as \eqref{eq:intro_growth}, to obtain faster concentration rates on the estimation error.
When the cost function $l$ is strongly convex and Lipschitz, the uniform stability is know to behave as $\epsilon = O(n^{-1})$ \citep[see][]{shalev2010learnability}. Yet necessary, this improvement is still vain since the rate in \eqref{eq:thm_improved_uniform_stability} would now be controlled by the sampling error in $\sqrt{n^{-1} \log(\delta^{-1})}$. 
\cite{shalev2009stochastic} asked if a bound on the estimation error for strongly convex and Lipschitz $l$ with rate in $O(n^{-1})$ is possible. Ten years later, this question was answered positively by \cite{klochkov2021stability} using an intricate decomposition of the estimation error and an extension of the McDiarmid's inequality so that \eqref{eq:thm_improved_uniform_stability} becomes in this case
\begin{equation} \label{eq:thm_improved_twice_uniform_stability}
\JJ(\hphi) - \JJ_* \leq c\frac{\log(n)}{n} \log\left(\frac{1}{\delta}\right),
\end{equation}
with high probability $1-\delta$ where $c > 0$ is a constant depending on the Lipschitz constant and the constant of strong convexity of $l$.
Up to the $\log(n)$ factor, this last result is in par with the optimal rate for the barycenter problem on Euclidean spaces \eqref{eq:optimal_rate_barycenter}.
However, the strong convexity and the Lipschitz assumptions on $l$ are rather restrictive and are not satisfied by the examples presented in Section \ref{sec:intro_applications}.

Probably one of the first work attempting to understand the fundamental mechanisms driving the concentration of the empirical minimizers is from \citet{van2017concentration} which was building on a former work of \citet{chatterjee2014new}.
They aim at bounding, with the notation of this paper, the excess risk $\ff(\hphi) - \ff(\phi_0) + \RR(\hphi)$ where $\phi_0 = \argmin \ff$ is the minimizer of the non-penalized version of \eqref{eq:opt_J}. It is again assumed that $\phi_0$ exists and is unique and that the penalty $\RR$ is convex.
This question is different from the one at stake in this work, since \eqref{eq:PAC_framework} only seeks to provide bounds for the random part, that is between $\hphi$ and $\phi_*$.
Relying on a curvature condition on the objective function around its minimizers, the analysis then uses uniform convergence techniques to provide the concentration bounds on the excess risk towards the expected excess risk.
The guarantees on $\|\hphi - \phi_0\|$ are then derived in the case where $\|\hphi - \phi_0\|^2 = \ff(\hphi) - \ff(\phi_0)$, which excludes many applications.

Recently, two works by \citet{schotz2019convergence} and \citet{ahidar2020convergence} have explored the concentration of $\vartheta(\hphi,\phi_*)$ using uniform convergence techniques too.
Both assume a growth of $\JJ$ around its minimizers of type \eqref{eq:intro_growth} without assuming the strong convexity of $l$.
In \citep{ahidar2020convergence}, the bounds are derived in the case where $l$ and $\Msp = \Xsp$ are bounded, exploiting a H\"older condition for $\phi \mapsto l(\phi,x)$ holding uniformly on $x$.
This hypothesis is usually accessible in the bounded setting, but might fail when the spaces are unbounded.
In \citep{schotz2019convergence}, the analysis is more general: $\Msp$ can be different from $\Xsp$ and $\Msp$, $\Xsp$ and $l$ can be unbounded. To deal with this additional difficulty, the author introduces an interesting weaker notion of smoothness for $l$ through a quadruple inequality, that can be interpreted as a H\"older condition on $\phi \mapsto l(\phi,x)$ with a constant that varies with $x$.
A similar condition will be used in the present work.
While providing a first significant understanding on the question, the bounds obtained depend on the entropy of the space $\Msp$, which might lead to rates suffering from the curse of dimensionality.
Indeed, in the example of the estimation of barycenters in non-positively curved spaces, the rates of \cite{schotz2019convergence,ahidar2020convergence} might be cursed by the dimension, and therefore suboptimal, since it is known that parametric rates hold regardless of the entropy of $\Msp$ \citep[see][]{le2022fast}.

\subsection{Contributions}

The review of the literature indicates that available results are insufficient to provide optimal concentration rates for $\vartheta(\argmin \hJJ ; \argmin \JJ)$ for a wide variety of estimation problems, such as those described in Section \ref{sec:intro_applications}.
This work attempts to fill that gap by providing a set of high-level assumptions from which optimal concentration rates can be proven. These assumptions fit most of the estimation problems encountered in application where the objective function can be non-convex and defined on unbounded sets.
Importantly, the rates obtained are parametric, and are not exponentially cursed by the complexity of the underlying space.

The paper is organized as follows. Notation and preliminary facts are gathered in the Section \ref{sec:preliminaries}. The Section \ref{sec:assumptions} introduces the different assumptions needed to prove the concentration rates that will be presented in Section \ref{sec:main_results}. These results will be proved in the Section \ref{sec:proof_main_result} while being applied to some practical estimation problems in Section \ref{sec:applications}.
To improve the readability of the core of the document, many technical results are postponed to the appendices.

\section{Preliminaries} \label{sec:preliminaries}

\subsection{General Facts on Random Variables}
For $p > 0$, the $p$-th norm of a random variable $X$ on the probability space $(\Xsp,\mu)$ is defined by $\| X \|_{L^p} = \E[|X|^p]^{1/p}$. This definition can be extended to the case $p = +\infty$ with the essential supremum $\| X \|_{L^\infty} = \esssup |X|$.

For $q \geq 1$, the $q$-exponential Orlicz norm of a random variable $X$ on the probability space $(\Xsp,\mu)$ is defined by $\| X \|_{\psi_q}$ as $\| X \|_{\psi_q} = \inf_{c > 0} \{ \E\left[ \exp( |X/c|^q) \right] \leq 2\}$.
The case $q = 1$ corresponds to sub-exponential random variables while $q = 2$ corresponds to sub-Gaussian ones. 
A random variable having a finite $\| X \|_{\psi_q}$ admits a tail satisfying $\P(|X| \geq t) \leq 2 \exp\left(-\frac{t^q}{\| X \|_{\psi_q}^q}\right)$ and its moments grow as $\|X\|_{L^p} \leq \| X \|_{\psi_q} p^{1/q}$.
These properties can be shown to be equivalent \citep[see][for more details]{vershynin2018high}.
Importantly, $\|X^2\|_{\psi_1} = \|X\|_{\psi_2}^2$ and $\|XY\|_{\psi_1} \leq \|X\|_{\psi_2} \|Y\|_{\psi_2}$ \cite[see][Lemma 2.7.6 and 2.7.7]{vershynin2018high}.

The $p$-th norm of vectors in $\R^d$ will be written using the usual notation $\|\cdot\|_p$. When $X \in \R^d$ is a random vector, its $\psi_q$ norm is defined by $\|X\|_{\psi_q} = \sup_{v \in \mathbb{S}^{d-1}} \| \langle X, v \rangle \|_{\psi_q}$. 
In particular, $\| \| X\|_2 \|_{\psi_q} \leq \sqrt{d} \| X \|_{\psi_q}$ \cite[see][Section 3.4]{vershynin2018high}.

Following \citet[Section 2.4]{boucheron2013concentration}, a random variable $X$ is said to be sub-gamma on the right tail with variance factor $\sigma^2$ and scale parameter $S$ if its logarithmic moment generating function satisfies
\begin{equation*}
  \log \E\left[ \exp\left(t(X - \E X) \right) \right] \leq \frac{t^2 \sigma^2}{2(1-St)},
\end{equation*}
for every $t \in \left(0,1/S\right)$. The set of all such random variables is denoted by $\mathrm{sub}\Gamma_+(\sigma^2,S)$.
The random variable $X$ is furthermore said sub-gamma on the left tail with variance factor $\sigma^2$ and scale parameter $S$ if $-X \in \mathrm{sub}\Gamma_+(\sigma^2,S)$. Similarly, the set of all such random variables is denoted by $\mathrm{sub}\Gamma_-(\sigma^2,S)$.
Lastly, a random variable $X$ is said to be sub-gamma with variance factor $\sigma^2$ and scale parameter $S$ if $X \in \mathrm{sub}\Gamma_+(\sigma^2,S) \cap \mathrm{sub}\Gamma_-(\sigma^2,S)$. The set of all such random variables is denoted by $\mathrm{sub}\Gamma(\sigma^2,S)$.
These sub-gamma properties can be characterized in terms of the tail. A random variable $X \in \mathrm{sub}\Gamma_+(\sigma^2,S)$ satisfies
\begin{equation*}
\P \left( X > \sqrt{2 \sigma^2 t} + S t\right) \leq \exp(-t),
\end{equation*}
with a similar tail on $-X$ for $X \in \mathrm{sub}\Gamma_-(\sigma^2,S)$. In particular, this tail condition implies that
\begin{equation*}
\P \left( X > t \right) \leq \exp\left(-\frac{-t^2}{2(\sigma^2 + St)}\right).
\end{equation*}

\subsection{Notation}

In this work, the function $l : \MM \times \Xsp \to \R$ is assumed to be such that $l(\cdot,\phi)$ are measurable and $\E[|l(X,\phi)|] < +\infty$ for every $\phi \in \MM$.

Let $(X_i)_{i=1}^n$ be random variables sampled independently from $\mu$. 
The vector of $\Xsp^n$ will be denoted with bold font, that is $\bX = (X_1, \ldots, X_n)$ for random vectors and $\bx = (x_1,\ldots, x_n)$ for vectors.
Let $Y$ be an independent copy of $X_1$. The random vector $\bX'$ is defined by $\bX' = (Y, X_2, \ldots, X_n)$. It is obtained by swapping the first element of $\bX$.

For convenient short hand notation, we will use $\Sc = \argmin \JJ$ while $\JJ_*$ will refer to its minimal value.
The empirical risk and its minimizers depend on the samples $\bX$. This dependency will not be made explicit but will be clear from the context.

Let $\mathcal{E}_1$ and $\mathcal{E}_2$ two subsets of $\Msp$, the notion of similarity between two sets used in this work is
\begin{equation*}
\vartheta(\mathcal{E}_1 ; \mathcal{E}_2) = \sup_{\phi_1 \in \mathcal{E}_1} \vartheta(\phi_1, \mathcal{E}_2) = \sup_{\phi_1 \in \mathcal{E}_1} \inf_{\phi_2 \in \mathcal{E}_2} \vartheta(\phi_1, \phi_2).
\end{equation*}
Note that $\vartheta( \cdot ; \cdot)$ is not a metric since it is not symmetric and $\vartheta(\mathcal{E}_1 ; \mathcal{E}_2) = 0$ only implies that $\mathcal{E}_1 \subseteq \mathcal{E}_2$. However, $\vartheta( \cdot ; \cdot)$ satisfies the triangle inequality, in the sense that $\vartheta(\mathcal{E}_1 ; \mathcal{E}_2) \leq \vartheta(\mathcal{E}_1 ; \mathcal{E}_3) + \vartheta(\mathcal{E}_3 ; \mathcal{E}_2)$ for all $\mathcal{E_1}, \mathcal{E_2}$, $\mathcal{E_3}$ non-empty subsets of $\Msp$. The reverse triangle inequality holds in the form $\left| \vartheta(\mathcal{E}_1 ; \mathcal{E}_2) -  \vartheta(\mathcal{E}_3 ; \mathcal{E}_2) \right| \leq \max \left( \vartheta(\mathcal{E}_1 ; \mathcal{E}_3), \vartheta(\mathcal{E}_3 ; \mathcal{E}_1) \right)$.

To further simplify the notation, we write $[\JJ \leq M] = \{ \phi \in \Msp \, | \, \JJ(\phi) \leq M\}$ and similar notation can be guessed from the context, for example $[\JJ \geq m] = \{ \phi \in \Msp \, | \, \JJ(\phi) \geq m\}$ and $[m \leq \JJ \leq M] = \{ \phi \in \Msp \, | \, m \leq \JJ(\phi) \leq M\}$.

A metric space $(\Msp,\vartheta)$ is said geodesic if for any two points $\phi,\psi \in \Msp$, there exists a path linking them whose length equals $\vartheta(\phi,\psi)$. For a proper introduction on these concepts see for example the paper of \citet{sturm2003probability}.

Let $(\Msp,\vartheta)$ be a geodesic space. It has positive curvature in the sense of Alexandrov, if for every constant speed geodesic $\gamma : [0,1] \to \Msp$ and every $\phi \in \Msp$ the concavity inequality holds:
\begin{equation} \label{eq:concavity_PC}
  \vartheta^2(\gamma(t),\phi) \geq (1-t) \vartheta^2(\gamma(0),\phi) + t \vartheta^2(\gamma(1),\phi) - t(1-t) \vartheta^2(\gamma(0),\gamma(1)).
\end{equation}
Similarly, $\Msp$ is said to be a non-positively curved space, in the sense of Alexandrov, if the converse inequality holds. The equality holds when $\Msp$ is a Hilbert space.

\section{Assumptions} \label{sec:assumptions}

In this section, the assumptions needed to derive the main concentration result are described.

\subsection{Existence of Minimizers}

Studying the concentration of the minimizers of $\JJ$ and $\hJJ$ is possible only if those exist. Their existence is formalized by the following assumption.

\begin{assumption} \label{ass:existstence_minimizer}
  The function $\JJ : \MM \to \R \cup \{ + \infty \}$ is proper, that is such that $\mathrm{dom} \JJ = \{ \phi \in \MM  \, | \, \JJ(\phi) \neq + \infty \} \neq \emptyset$ and have at least one minimizer.
  Furthermore, for all $n \geq 1$ and $\mu^n$-almost every $\bx \in \Xsp^n$, the empirical risks $\hJJ$, associated to $\bx$, are also assumed to be proper with at least one minimizer. Its minimizers, as functions of $\bx$, are all assumed to be measurable.
\end{assumption}
On a metric space $(\Msp, \vartheta)$ satisfying the Heine–Borel property (that is for which any closed and bounded subspace is compact), at least one minimizer exists when the function $\JJ$ (respectively $\hJJ$) is lower semi-continuous and either $\MM$ is bounded or $\JJ$ (respectively $\hJJ$) is coercive.

The measurability assumption of the minimizers is made upfront to avoid technical considerations about their existence, which will be clear in usual applications.
It can be shown, under mild conditions that there will be at least one measurable minimizer in $\hSc$ \citep[see][]{hess1996epi}. 
Moreover, the measurability assumptions can be weakened using outer expectations \citep[see][]{van1997weak}.

\subsection{Local H\"older-type Error Bound}

The following assumption describes the growth of $\JJ$ locally around its minimizers and is a generalization of \eqref{eq:intro_growth}.

\begin{assumption} 

\label{ass:local_holder_error}
There exist $\JJ_0 \in (\JJ_*,+\infty]$, $\beta \geq 1$ and $\tau > 0$ such that for all $\phi \in \MM$,
\begin{equation} \label{eq:lojasiewicz}
  \phi \in [\JJ \leq \JJ_0] \implies \tau \vartheta(\phi, \Sc)^\beta \leq \mathcal{J}(\phi) - \mathcal{J}_*.
\end{equation}
\end{assumption}

The bound \eqref{eq:lojasiewicz} is a central ingredient in this work as it avoids pathological oscillations of $\JJ$ around its minimizers and it provides leverage to use concentration techniques on the function values $\JJ(\widehat{\phi}) - \JJ_*$ to bound $\vartheta$.

It is important to mention that many functions used in applications satisfy Assumption \ref{ass:local_holder_error}.
For example, any strongly convex function $\JJ$ satisfies the assumption with $\beta = 2$ and $\JJ_0 = + \infty$. 
The assumption can still hold even when the function is not convex. 
An argument in this direction is that when $\Msp = \R^d$ and $\JJ$ is a semi-algebraic function \citep{bochnak2013real}, the bound \eqref{eq:lojasiewicz} holds on any compact subset \cite[Theorem II]{lojasiewicz1993geometrie}. 
Semi-algebraic functions contains for examples piecewise polynomials, square root, quotients, norms, relu, rank norm to name a few. As a consequence of Tarski Seidenberg theorem \cite[Chapter 2]{basu2014algorithms}, these functions are furthermore stable under many operations such as differentiation or composition.
For these reasons, the Assumption \ref{ass:local_holder_error} covers most of the functions encountered in applications where $\Msp = \R^d$.
Extension of this result for semi-analytic functions to Riemannian analytic manifolds can be found in \cite[Theorem 6.4]{bierstone1988semianalytic}.

This assumption appears in many communities under different names such as local H\"older-type bound in functional analysis \citep{li2013global}, \L{}ojasiewicz inequality in real algebraic geometry \citep{lojasiewicz1993geometrie} or \L{}ojasiewicz error bound inequality in optimization \citep{bolte2017error}.
In the particular case of Fréchet means where $\Xsp = \Msp$ and $l = \rho^2$, the Assumption \ref{ass:local_holder_error} is known as variance inequality \citep{sturm2003probability}. 
It is also called a margin condition in the works of \citet{barrioLecturesEmpiricalProcesses2007} and of \citet{van2017concentration} or a low noise assumption in the one of \citet{ahidar2020convergence}. 

In the ERM literature, the Bernstein condition is often presented as central to derive fast concentration rates \citep[see for example][]{koltchinskii2006local,bartlett2006empirical,klochkov2021stability} which assumes that there exists a constant $c$ such that for all $\phi \in \Msp$, there exists a $\phi_* \in \Sc$ such that $\E[ (l(\phi, X) - l(\phi_*,X))^2] \leq c (\JJ(\phi) - \JJ_*)$.
This condition is in particular verified when $\JJ$ satisfies the Assumption \ref{ass:local_holder_error} for $\beta = 2$, $\JJ_0 = +\infty$ and that $l(\cdot, X)$ is uniformly Lipschitz.

\begin{remark} \label{rmk:lojasiewicz_modulus}
The Assumption \ref{ass:local_holder_error} is equivalent to assuming that there exist $\JJ_0 \in (\JJ_*,+\infty]$ and $\beta > 0$ with the quantity
\begin{equation*}
\tau = \inf \left\{  \frac{\JJ(\phi) - \JJ_*}{\vartheta(\phi,\Sc)^\beta}, \quad\phi \in [\JJ \leq \JJ_0]\right\}
\end{equation*}
being positive. 
Such a quantity might however be delicate to compute and only a positive lower-bound of $\tau$ can be accessible. 
The constant $\beta^{-1}$ is often referred to as the \L{}ojasiewicz exponent in the optimization literature \citep{bolte2017error}. To the best of the author's knowledge, there is no particular name for $\tau$, so it will be named \L{}ojasiewicz constant of $\JJ$.
\end{remark}

\subsection{Convergence to the Basin} \label{sec:convergence_bassin}

To control $\vartheta(\hSc ; \Sc)$ leveraging Assumption \ref{ass:local_holder_error}, the minimizers of $\hJJ$ have to fall in the level-set $[\JJ \leq \JJ_0]$. 
This event will be assumed to hold with high probability.

\begin{assumption} \label{ass:convergence_bassin}
  There exists $\eta : \N \to [0,1]$  with $\lim_{n \to + \infty}\eta(n) = 0$ such that
  \begin{equation*}
    \P \left( \sup_{\hphi \in \hSc} \JJ(\hphi) > \JJ_0 \right) \leq \eta(n).
  \end{equation*}
\end{assumption}

This assumption is rather natural since, as the number of points increases, the empirical risk $\hJJ$ is expected to better approach the true risk. 
The value of $\JJ(\hphi)$ is therefore expected to converge to the minimum $\JJ_*$ so that $\hphi$ would fall into the level-set $[\JJ \leq \JJ_0]$.

The technique used to derive an Assumption \ref{ass:convergence_bassin}, that is optimal, depends on the context and is therefore left to be checked for each application. 
Let us however mention some techniques that can be helpful.
Following \eqref{eq:estimation_error_decomposition}, 
if $\sup_{\phi \in \Msp}\Delta(\phi) \leq t$ holds with high probability, for example when the Rademacher complexity of the function class or the entropy of $\Msp$ are bounded, then setting $t = (\JJ_0 - \JJ_*)/2$ would imply the Asumption \ref{ass:convergence_bassin}.
When concentration-based argument are not available, for example when the diameter of $\Msp$ is unbounded, the small-ball method \citep{mendelson2014learning,mendelson2018learning} might be used to prove that Assumption \ref{ass:convergence_bassin} holds.
Finally, the convergence to the basin might also be proved considering approaches like the one of \citet{schotz2019convergence}.

\subsection{\L{}ojasiewicz Constant of \texorpdfstring{$\hJJ$}{the empirical risk}} \label{sec:ass_lojasiewicz}

Let $\bX \in \Xsp^n$ and $\by \in \supp(\mu)^n$. Let $\tJJ$, respectively $\tSc$, refer to the empirical risk and the empirical set of minimizers associated to $\by$.
Let $\htau$ be the random variable defined by
\begin{equation*}
  \begin{aligned}
\htau = \inf_{\by} \inf \left\{  \frac{\hJJ(\phi) - \hJJ_*}{\vartheta(\phi,\hSc)^\beta}, \quad \phi \in \left( [\JJ \leq \JJ_0] \cap \tSc  \right) \right\}.
\end{aligned}
\end{equation*}
The constant $\htau$ quantifies the behavior of $\hJJ$, around its minimizers, regarding only the minimizers (falling into the basin $[\JJ \leq \JJ_0]$) of all possible empirical risks.

The next assumption will be crucial in proving a notion of stability of the empirical minimizers.

\begin{assumption} \label{ass:lojasewicz_modulus}
There exists $\kappa : \N \to [0,1]$ with $\lim_{n \to + \infty}\kappa(n) = 0$ such that
 \begin{equation*}
  \P \left( \htau < \frac{\tau}{2} \right) \leq \kappa(n).
  \end{equation*}
\end{assumption}

\begin{remark}
  Similarly to Remark \ref{rmk:lojasiewicz_modulus}, it might only be possible to compute a lower bound to $\htau$. In this assumption, the lower bound however has to be sufficiently large to be greater than $\frac{1}{2} \tau$ with high probability.
  Actually, the only hypothesis needed is that the lower-bound is larger than a constant with high probability.
  The constant here is chosen to be $\tau/2$ in order not to overload the bounds of Theorem \ref{thm:concentration}.
\end{remark}

\begin{remark}
  The Assumption \ref{ass:lojasewicz_modulus} is weaker than supposing that
  \begin{equation*}
    \inf \left\{  \frac{\hJJ(\phi) - \hJJ_*}{\vartheta(\phi,\hSc)^\beta}, \quad\phi \in [\JJ \leq \JJ_0]\right\} \geq \frac{1}{2}\tau,
  \end{equation*}
  with high probability, which involves the infimum on the whole level-set. The proof of the main result only requires the local error-bound \eqref{eq:lojasiewicz} to hold for $\hJJ$ regarding only the minimizers of $\tJJ$.
\end{remark} 

A general treatment of the phenomenon assumed in the Assumption \ref{ass:lojasewicz_modulus} and providing the weakest conditions under which it holds is, to the knowledge of the author, an open question and is postponed to future works.
However, a systematic analysis shows, in the examples, that this assumption is verified. 
Note that in some applications, $\ff$ might be convex and a regularization $\RR$ is added in order to ensure the uniqueness of the minimizer and to improve the resolution of \eqref{eq:opt_J} with numerical methods. This regularization is usually strongly convex and thus drives the \L{}ojasiewicz constants of $\JJ$ and $\hJJ$.

\subsection{Smoothness of \texorpdfstring{$l$}{l}}
The concentration phenomenon can occur when the integrand $l$ possesses some smoothness. The smoothness needed is described in the following assumption.

\begin{assumption} \label{ass:holder_l}
There exist $\alpha > 0$, a pseudometric\footnote{A pseudometric on $\Xsp$ is a function from $\Xsp \times \Xsp \to \R_+$ satisfying the same axioms of a metric except the implication $a(x,y) = 0 \implies x=y$.}  $a : \Xsp \times \Xsp \to \R_+$ and a subset $\Upsilon \subseteq \Msp$ such that
\begin{equation} \label{eq:holder_l}
  \phi, \psi \in \Upsilon \implies l(\phi,x) - l(\psi,x) - l(\phi,y) + l(\psi,y) \leq a(x,y) \vartheta(\phi,\psi)^{\alpha},
\end{equation}
for $\mu$-almost every $x, y \in \Xsp$. 
The function $a$ should further satisfy $\| a(X,Y) \|_{\psi_1} < + \infty$ for independent $X,Y \sim \mu$.

The set $\Upsilon$ should contain $\Sc$ and it is further assumed that there exists $\iota : \N \to [0,1]$ with $\lim_{n \to + \infty} \iota(n) = 0$ such that $\P(\hSc \not\subset \Upsilon) \leq \iota(n)$.
\end{assumption}

The assumption made here can be interpreted as a H\"older condition on $l(\cdot,x)$ with a constant that may vary with $x$. Since $\Xsp$ is a metric space, this constant has to be measured with respect to a reference point. 
This translates to assuming a H\"older condition on $l(\cdot,x) - l(\cdot,y)$ with constant $a(x,y)$ with the integrability condition that $a(X,Y)$ should be sub-exponential.
In the remaining of the paper, the notation $\| a \|_{\psi_1}$ will be a shorthand to the quantity $\| a(X,Y) \|_{\psi_1}$ for independent $X,Y \sim \mu$. Note that $\| a \|_{\psi_1}$ is oftentimes called the sub-exponential diameter of the space $(\Xsp, a, \mu)$ \citep{maurer2021concentration}.

Remark that when $l(\cdot, x)$ satisfy a H\"older condition, that is when there exist $K > 0$ and $\alpha > 0$ such that for all $x \in \Xsp$, $|l(\phi,x) - l(\psi,x)| \leq K \vartheta(\phi,\psi)^\alpha$ for all $\psi, \phi \in \Msp$ then Assumption \ref{ass:holder_l} holds with $a(x,y) = 2K$ and $\iota = 0$. Such a bound is oftentimes available when $\diam(\Xsp) < + \infty$.

The inequality in \eqref{eq:holder_l}, when holding on the whole space $\Upsilon = \Msp$, is also called a \emph{quadruple inequality} in \citep{schotz2019convergence}.
In the Fréchet mean setting, that is when $\Msp = \Xsp$ and $l = \rho^2$ and if this condition holds for all $\phi,\psi \in \Msp$ with $a(x,y) = 2 \rho(x,y)$, then Assumption \ref{ass:holder_l} characterizes metric spaces with non-positive curvature \cite[Definition 2.1]{sturm2003probability}. 

\begin{remark}
  The class of pseudometric functions includes distances and additively separable functions of the form $a(x,y) = \left(h(x) + h(y) \right)\indic_{x \neq y}$ with $h : \Xsp \to \R_+$.
\end{remark}

\section{Main Results} \label{sec:main_results}

It is now time to present the main results of this work.
The concentration in probability of $\vartheta(\hSc;\Sc)$ will be derived leveraging the following result.

\begin{theorem} \label{thm:mcdiarmid_combes}
Let $f : \Xsp^n \to \R$ and $\Bc \subseteq \Xsp^n$ such that $p = \P(\bX \notin \Bc) \leq 3/4$.
Assume there exist a pseudometric $b : \Xsp \times \Xsp \to \R^+$ with $\|b\|_{\psi_1} < + \infty$ such that
\begin{equation} \label{eq:mcdiarmid_condition}
  |f(\bx)-f(\by)| \leq \sum_{i=1}^n b(x_i,y_i), \quad \text{for all }\bx,\by \in \Bc.
\end{equation}  
Then 
\begin{equation*}
  f(\bX) - \E\left[f \, | \, \bX \in \Bc\right] \leq 4 n\| b \|_{\psi_1} \sqrt{p} + e\|b\|_{\psi_1} \left(2\sqrt{n\log\left(\frac{1}{\delta}\right)} + \log\left(\frac{1}{\delta}\right)\right),
\end{equation*} 
with probability at least $1-p-\delta$.
\end{theorem}

This result is a generalization of McDiarmid's inequality when the differences are sub-exponential with high probability. 
In the literature, the derivation of such inequalities when differences are assumed bounded with high-probability has been studied in a few works \citep[for example][]{mcdiarmid1998concentration,kutin2002extensions,kutin2012almost,combes2015extension,warnke2016method}.
This result is a combination of ideas from \citet{maurer2021concentration} and \citet{combes2015extension} for which a proof is given in Section \ref{sec:proof_main_result}.

In the present case of the concentration of $\vartheta(\hSc;\Sc)$, the subset $\Ac \subseteq \Xsp^n$ of interest, that will be used as $\Bc$ in the theorem, will be defined by 
\begin{equation} \label{eq:def_Ac}
\Ac = \left\{ \sup_{\hphi \in \hSc}\JJ(\hphi) \leq \JJ_0 \right\} \cap \left\{ \htau \geq \frac{1}{2} \tau \right\} \cap \left\{ \hSc \subseteq \Upsilon \right\}.
\end{equation} 

The remainder of this section consists in computing the conditional expectation and checking the condition \eqref{eq:mcdiarmid_condition}.
Both will be achieved by taking advantage of the \emph{stability} of the empirical minimizers in $\Ac$, in the sense presented in the next paragraph. 
In the rest of this section, the probability of being outside of $\Ac$ will be referred to as $p_n = \P(\bX \notin \Ac)$ and is controlled using Assumptions \ref{ass:convergence_bassin}, \ref{ass:lojasewicz_modulus} and \ref{ass:holder_l} by $p_n \leq \eta(n) + \kappa(n) + \iota(n)$.

\subsection{The Empirical Minimizers are Stable}

The cornerstone of the results obtained in this work is the stability of the empirical minimizers of $\hJJ$ when the components of $\bX$ are perturbed. Let $\bX, \bY \in \Xsp^n$. Conditionally on the event $\bX \in \Ac, \bY \in \Ac$, the minimizers in $\hSc$ and $\tSc$ fall into the level set $[\JJ \leq \JJ_0]$. Furthermore, conditionally on this event, the \L{}ojasiewicz constant $\htau \geq \frac{1}{2} \tau$ so that, for a minimizer $\tphi \in \tSc$,
\begin{equation*}
  \begin{aligned}
    \frac{\tau}{2} \vartheta(\tphi,\hSc)^{\beta} & \leq \hJJ(\tphi) - \hJJ_* \\
    & \leq \inf_{\phi \in \hSc} \hJJ(\tphi) - \hJJ(\phi) + \tJJ(\phi) - \tJJ(\tphi) & \text{($\tphi$ is a minimizer of $\tJJ$)}\\
    & = \inf_{\phi \in \hSc} \hff(\tphi) - \hff(\phi) + \tff(\phi) - \tff(\tphi) & \text{(the terms in $\RR$ cancel out)} \\
    & = \inf_{\phi \in \hSc} \frac{1}{n} \sum_{i=1}^n l(\tphi,X_i) - l(\phi,X_i) + l(\phi,Y_i) - l(\tphi,Y_i) \\
    & \leq \frac{1}{n} \sum_{i=1}^n a(X_i,Y_i) \vartheta(\tphi,\hSc)^{\alpha} .& \text{(from Assumption \ref{ass:holder_l})}
  \end{aligned}
  \end{equation*}
Since this bound holds for any $\tphi \in \tSc$, the following lemma is derived.

\begin{lemma} \label{lem:hJJ_error_bound}
For $\beta > \alpha $ and conditionally on the event $\bX \in \Ac, \bY \in \Ac$, it holds that
\begin{equation*}
 \vartheta(\tSc ; \hSc)^{\beta - \alpha} \leq \frac{2}{\tau} \frac{1}{n} \sum_{i=1}^n a(X_i,Y_i).
\end{equation*}
\end{lemma}

This notion of stability of the empirical minimizers is, in essence, similar to the notions of stability of \cite{bousquet2002stability}.
Since the probability $p_n$ is controlled by Assumptions \ref{ass:convergence_bassin}, \ref{ass:lojasewicz_modulus} and \ref{ass:holder_l}, the stability happens with high-probability, and echoes the work of \cite{kutin2012almost}.

This notion of stability can then be exploited to bound the expectation of $\vartheta(\hSc ; \Sc)^\beta$ conditionally on $\bX \in \Ac$.

\begin{lemma} \label{lem:bound_conditional_expectation}
  Under Assumptions \labelcref{ass:existstence_minimizer,ass:local_holder_error,ass:convergence_bassin,ass:lojasewicz_modulus,ass:holder_l} with the further condition that $p_n \leq 3/4$ then
  \begin{equation*}
    \E \left[ \vartheta( \hSc ; \Sc )^{\beta} \, | \,  \bX \in \Ac \right] \leq L^{\frac{\beta}{\beta-\alpha}} \left( c_1 n^{-\frac{\alpha}{\beta-\alpha}} + K \sqrt{p_n} \right),
  \end{equation*}
  with 
  \begin{equation*}
  \begin{aligned}
    L &= \tau^{-1} \| a \|_{\psi_1}, & K &= 2^{\max(0,\alpha-1)+2} \left( L^{\frac{-\alpha}{\beta-\alpha}} \diam(\Sc)^{\alpha}+c_2\right),\\
    c_1 &= c_1(\alpha,\beta) = \left(\frac{4\beta}{\beta-\alpha}\right)^\frac{\beta}{\beta-\alpha}, & c_2 &= 
  c_2(\alpha,\beta) = 2 \max\left(\frac{4\alpha}{\beta-\alpha},1\right)^\frac{\alpha}{\beta-\alpha}.
  \end{aligned}
  \end{equation*}

\end{lemma}

Similarly to \eqref{eq:estimation_error_decomposition_stability}, the proof is essentially based on switching the variables $Y$ and $X_1$ in the expectation. This step allows to leverage the stability of the empirical minimizers of Lemma \ref{lem:hJJ_error_bound}.
However, due to the conditioning on $\bX \in \Ac$, the random variables $Y$ and $X_i$ do not share the same distribution anymore so that extra care must be taken. 
Note that the function $l$ or the spaces $\Msp$ or $\Xsp$ can be unbounded which require additional precautions compared to the usual argument of \citet{bousquet2002stability,feldman2018generalization,feldman2019high,bousquet2020sharper}.
For these reasons, the proof of this lemma is postponed to Section \ref{sec:proof_main_result:conditional_expectation}.

The next corollary is an instance of Lemma \ref{lem:bound_conditional_expectation} for particular values of $\alpha$ and $\beta$.
\begin{corollary}
  When $\beta = 2$ and $\alpha = 1$, the bound of Lemma \ref{lem:bound_conditional_expectation} reads
  \begin{equation*}
    \E \left[ \vartheta( \hSc ; \Sc )^{\beta} \, | \,  \bX \in \Ac \right] \leq 2^6 L^2 n^{-1} + 4 \left(L \diam(\Sc) + 8 L^2\right) \sqrt{p_n}.
  \end{equation*}
  When $\Sc$ contains a unique minimizer, the bound further simplifies to
  \begin{equation*}
    \E \left[ \vartheta( \hSc ; \Sc )^{\beta} \, | \,  \bX \in \Ac \right] \leq 2^6 L^2 n^{-1} + 2^5 L^2 \sqrt{p_n}.
  \end{equation*}
\end{corollary}

\subsection{Concentration in Probability}

The main results of this paper are now stated.
\begin{theorem} \label{thm:concentration}
  Suppose Assumptions \labelcref{ass:existstence_minimizer,ass:local_holder_error,ass:convergence_bassin,ass:lojasewicz_modulus,ass:holder_l} hold and assume furthermore that $p_n \leq 3/4$.
  For $0 < \beta -\alpha < 2$, let $q = \min(\beta - \alpha,1)$, $Q = \max(\beta-\alpha,1)$ and $s = \min\left( 1, \frac{2}{\beta-\alpha}-1\right)$. Then, for any $\delta > 0$,
  \begin{equation} \label{eq:concentration_rate}
  \begin{aligned}
    \vartheta(\hSc ; \Sc)^q 
    & \leq L^{\frac{1}{Q}} \left( c_1^{\frac{q}{\beta}} n^{-\frac{\alpha}{\beta Q}} + 2^{\frac{1}{Q}} e \left( 2\sqrt{\frac{1}{n^s} \log\left(\frac{1}{\delta}\right)} + \frac{1}{n^s}\log\left(\frac{1}{\delta}\right) \right) \right) + C_n p_n^{\frac{q}{2\beta}},
  \end{aligned}
\end{equation}
with probability at least $1 - p_n - \delta$ where $C_n = L^{\frac{1}{Q}} \left( K^{\frac{q}{\beta}} + 2^{2+\frac{1}{Q}} n^{1 - \frac{1}{Q}}\right)$ and
where the constants $c_1$, $K$ and $L$ are defined in Lemma \ref{lem:bound_conditional_expectation}.
\end{theorem}

\begin{proof}
  Let $f(\bX) = \vartheta(\hSc ; \Sc)^{q}$. 
  The McDiarmid's inequality of Theorem \ref{thm:mcdiarmid_combes} is applied to bound $f(\bX) - \E[f \,|\, \bX \in \Ac] $ with high probability. 

  The conditional expectation will be bounded using Lemma \ref{lem:bound_conditional_expectation} as
  \begin{equation*}
  \begin{aligned}
    \E\left[ f(\bX) \, | \, \bX \in \Ac \right] & \leq \E\left[ \vartheta(\hSc ; \Sc)^{\beta} \,|\, \bX \in \Ac \right]^{\frac{q}{\beta}} & (\text{using Jensen's inequality, } q \leq \beta)\\
    & \leq L^{\frac{1}{Q}} \left( c_1^{\frac{q}{\beta}} n^{-\frac{\alpha}{\beta Q}} + K^\frac{q}{\beta} p_n^\frac{q}{2\beta} \right), & (\text{since } t > 0 \mapsto t^\frac{q}{\beta} \text{ is subadditive})
    \end{aligned}
  \end{equation*}
  where we also have used the fact that $q/(\beta - \alpha) = 1/Q$.
  
 The remaining of the proof consists in constructing the pseudometric $b$ satisfying the conditions of Theorem \ref{thm:mcdiarmid_combes}.
 With a slight abuse of notation let $\hSc$ and $\tSc$ be the sets of minimizers of $\hJJ$ defined from $\bx$ and $\tJJ$ from $\by$ respectively.
 When $\bx, \by \in \Ac$, 
 \begin{equation*}
 \begin{aligned}
 |f(\bx) - f(\by)| & \leq \left| \vartheta( \hSc ; \Sc)^{q} - \vartheta(\tSc ; \Sc)^{q} \right|  \leq \max\left( \vartheta( \hSc ; \tSc)^{q}, \vartheta( \tSc ; \hSc)^{q}\right) \leq \sum_{i=1}^n \left( \frac{2}{n \tau} a(x_i,y_i) \right)^\frac{1}{Q},
    \end{aligned}
  \end{equation*}
  using the triangle inequality combined with Lemma \ref{lem:hJJ_error_bound}.
  Therefore, the choice $b = \left( \frac{2}{n \tau} a \right)^{\frac{1}{Q}}$ satisfies the assumptions of Theorem \ref{thm:mcdiarmid_combes} so that
  \begin{equation*}
  \begin{aligned}
    f(\bX) - \E[f \,|\, \bX \in \Ac] & \leq 2^{2 + \frac{1}{Q}} L^{\frac{1}{Q}} n^{1-\frac{1}{Q}} \sqrt{p_n} + 2^{\frac{1}{Q}} e L^{\frac{1}{Q}} n^{-\frac{1}{Q}}\left(2\sqrt{n\log\left(\frac{1}{\delta}\right)} + \log\left(\frac{1}{\delta}\right) \right) \\
    & \leq 2^{2+\frac{1}{Q}} L^{\frac{1}{Q}} n^{1-\frac{1}{Q}} \sqrt{p_n} + 2^{\frac{1}{Q}} e L^{\frac{1}{Q}} \left(2\sqrt{\frac{1}{n^s} \log\left(\frac{1}{\delta}\right)} + \frac{1}{n^s} \log\left(\frac{1}{\delta}\right) \right),
  \end{aligned}
\end{equation*}
with probability at least $1 - p_n - \delta$ with $s = \min\left(1,\frac{2}{\beta-\alpha}-1\right)$. 
This bound combined with the bound on $\E\left[ f(\bX) \, | \, \bX \in \Ac \right]$ together with the fact that $\frac{q}{2\beta} < \frac{1}{2}$ proves the claimed result.
\end{proof}

The notable case $\beta=2$ and $\alpha = 1$ is a direct consequence of this result.
\begin{corollary} \label{cor:concentration}
When $\beta = 2$ and $\alpha = 1$, the bound of Theorem \ref{thm:concentration} reads
  \begin{equation*}
  \begin{aligned}
    \vartheta(\hSc ; \Sc) 
    &\leq 8 L \left( n^{-\frac{1}{2}} + 2\sqrt{\frac{1}{n}\log\left(\frac{1}{\delta}\right)}+ \frac{1}{n} \log\left(\frac{1}{\delta}\right) \right) + C p_n^{\frac{1}{4}},
  \end{aligned}
\end{equation*}
with probability at least $1 - p_n - \delta$, and $C = \left( 14 L + 2\sqrt{L \diam(\Sc)} \right)$.
\end{corollary}

The following comments can be drawn from these results. 
As long as $\beta$ and $\alpha$ are independent of the complexity of $\Msp$, the rate of convergence will not depend exponentially on the dimension of $\Msp$.
This comes at the expense of additional terms driven by the probability of not falling into the basin $\Ac$ where the empirical minimizers are stable.
When $p_n$ decays sufficiently fast, the quantity $n^{1-\frac{1}{Q}} \sqrt{p_n}$ goes to zero faster than $n^{-s}$ and becomes negligible for large enough $n$. Actually, when $\beta-\alpha \leq 1$, the quantity $n^{1-\frac{1}{Q}} = 1$ so that $\sqrt{p_n}$ should decay faster than $n^{-s} = n^{-\frac{1}{2}}$ or $n^{-\frac{\alpha}{\beta Q}} = n^{-\frac{\alpha}{\beta}}$ which is a rather mild condition since in the considered applications, $p_n$ will decay as $\exp(-cn)$ for some $c > 0$.

When $\beta - \alpha$ approaches 2, the rate deteriorates. It is not clear to the author if this phenomena is to be excepted or is an artifact introduced by the tools used in the proofs. 
On the one hand, this can be due to a parameter $\beta$ getting larger meaning that the functions $\JJ$ and $\hJJ$ are becoming flatter, which naturally slows down the rate. On the other hand, this could also be interpreted as a lack of regularity of the problem since $\alpha$ would be too small compared to $\beta$.
In case the rate becomes too slow, techniques like \citep{schotz2019convergence} might be able to provide a usable rate.

Arguably the most studied case is the one where $\Msp$ is a Hilbert space and the losses $l(\cdot, x)$ are bounded, $\omega$-strongly convex and $K$-Lipschitz uniformly for $\mu$-almost every $x \in \Xsp$.
This case was discussed in Section \ref{sec:review_literature}. A combination of \eqref{eq:intro_growth} and the Proposition 2.1 of \citet{klochkov2021stability} was showing that, for all $\delta > 0$, the rate
\begin{equation} \label{eq:klochkov}
\vartheta(\hphi, \phi_*)^2 \leq c \frac{K^2}{\omega^2} \frac{\log(n)}{n} \log\left( \frac{1}{\delta}\right),
\end{equation}
was holding with high probability $1 - \delta$.
Such functions guarantee that the Assumptions \labelcref{ass:existstence_minimizer,ass:local_holder_error,ass:convergence_bassin,ass:lojasewicz_modulus,ass:holder_l} hold for $\tau = \omega$, $\beta = 2$, $\alpha = 1$, $\|a\|_{\psi_1} = 2K$, $\JJ_0 = + \infty$ and $\eta = \iota = \kappa = 0$. In view of Corollary \ref{cor:concentration}, a bound
\begin{equation} \label{eq:strgly_cvx_lipschitz}
\vartheta(\hphi, \phi_*)^2 \leq c \frac{K^2}{\omega^2} \frac{1}{n} \left(1 + \log\left( \frac{1}{\delta} \right) + \frac{1}{n} \log\left( \frac{1}{\delta} \right)^2 \right),
\end{equation}
holding with high probability $1 - \delta$ is obtained for all $\delta > 0$. Note that the constants $c > 0$ in the two bounds are different.
Up to the logarithm factor, the bound \eqref{eq:klochkov} is a parametric, while \eqref{eq:strgly_cvx_lipschitz} gives a true parametric rate. The bounds also differ in their nature, \eqref{eq:klochkov} is a 0-concentration bound while \eqref{eq:strgly_cvx_lipschitz} is a mean-concentration one. 
Finally, the bound \eqref{eq:strgly_cvx_lipschitz} contains an extra term $n^{-1} \log(\delta^{-1})^2$ which inherits from the sub-exponential nature of $\| a \|_{\psi_1}$.
With the further assumption that the losses are $K$-Lipschitz, it is possible to use a standard McDiarmid inequality, since in this case Lemma \ref{lem:hJJ_error_bound} would give $\vartheta(\hSc ; \tSc) \leq K n^{-1} \sum_{i=1}^n \indic_{X_i \neq Y_i}$. This strategy would remove the extra term $n^{-1} \log(\delta^{-1})^2$ of \eqref{eq:strgly_cvx_lipschitz}.

  The case $\beta = 2$ and $\alpha = 1$ is particularly important in modern statistical learning because: (i) the losses $l$ have bounded sub-gradients in order to use first-order descent algorithms and (ii) the risks are designed to have a strong local convexity to accelerate the convergence of optimization algorithms. 
  The reason why the case $\beta = 2$ is often encountered can also be justified in view of a Taylor expansion of the risk (assuming that $\Msp$ is a Riemannian manifold, proper smoothness of $\JJ$ and  the uniqueness of its minimizers). Sufficiently close to the minimizers, the quadratic growth driven by the Hessian of $\JJ$, if not vanishing, dominates the higher-order terms. This leads to local quadratic error bounds.
  Estimation problems where $\beta \neq 2$ and $\alpha \neq 1$ can be constructed but are usually obtained by designing pathological settings. 
  Such examples can be found, for instance, in the literature of power Fréchet means in Hadamard spaces \citep{schotz2023variance,yun2023exponential}. 
  Moreover, in the literature of Fréchet means on positively curved spaces, it is possible to design ad-hoc measures $\mu$ to obtain flatter local error bounds, which is often called smeariness. An example on the sphere was found by \citet{eltzner2019smeary} where the Hessian of the risk vanishes at the minimizer, leading to higher order polynomials of the Taylor expansion to drive the growth around the minimizer.

Finally, the optimality of the rates Theorem \ref{thm:concentration} is difficult to assess because of its level of generality. However, its application to various scenarii provides the optimal parametric rates even for complex underlying spaces $\Msp$.

\subsection{Rates in Expectation}

When $p_n = 0$, the Lemma \ref{lem:bound_conditional_expectation} used to prove Theorem \ref{thm:concentration} also provides a bound on the expectation of $\vartheta^\beta(\hSc ; \Sc)$ of kind
\begin{equation*}
  \E \left[ \vartheta(\hSc ;\Sc)^\beta \right] \leq c_1 \left(\tau^{-1} \|a\|_{\psi_1} \right)^{\frac{\beta}{\beta-\alpha}} n^{-\frac{\alpha}{\beta - \alpha}},
\end{equation*} 
with $c_1 = c_1(\alpha,\beta)$ defined in Lemma \ref{lem:bound_conditional_expectation}.
Additional developments are however needed to obtain a bound on the expectation $\E\left[ \vartheta(\hSc ;\Sc)^r \right]$ for $r > \beta$ or whenever $p_n > 0$.
The total law of expectation gives
\begin{equation} \label{eq:total_law_expectation}
  \begin{split}
  \E\left[ \vartheta(\hSc;\Sc)^r \right] 
  & \leq \E \left[ \vartheta(\hSc;\Sc)^r  \, | \, \bX \in \Ac \right] + \E \left[ \vartheta(\hSc;\Sc)^r \indic_{\bX \notin \Ac} \right],
  \end{split}
\end{equation}
so that a bound can be obtained when the expectation of $\vartheta^r \indic_{\bX \notin \Ac}$ can be bounded. 
On the other hand, the expectation of $\vartheta^r$ conditionally on $\bX \in \Ac$ for $r > \beta$ can be obtained by integrating its tail \eqref{eq:concentration_rate}.
The complete computations are omitted in this paper as the applications detailed in Section \ref{sec:applications} will not heavily rely on this result.

\section{Proof of the Main Results} \label{sec:proof_main_result}

It is now time to prove the supporting results needed for the derivation of Theorem \ref{thm:concentration}.

\subsection{General Concentration Results} \label{sec:proof_main_result:general_concentration}

The first result is adapted from the work of \citet{maurer2021concentration}, a generalization of McDiarmid's inequality.

\begin{definition}[{\citet[Definition 1]{maurer2021concentration}}]
  Let $f : \Xsp^n \to \R$, a $\bx \in \Xsp^n$ and a random $\bX \in \Xsp^n$ with its components sampled independently.
  Let $\bY \in \Xsp^n$ an independent copy of $\bX$.
  The $k$-th centered conditional version of $f$ is the random variable
  \begin{equation*}
    \begin{aligned}
    f_k(\bX)(\bx) & = f(x_1,\ldots,x_{k-1},X_k,x_{k+1},\ldots,x_n) - \E\left[ f(x_1,\ldots,x_{k-1},Y_k,x_{k+1},\ldots,x_n) \right] \\
    &= \E\left[ f(x_1,\ldots,x_{k-1},X_k,x_{k+1},\ldots,x_n) - f(x_1,\ldots,x_{k-1},Y_k,x_{k+1},\ldots,x_n) | X_k \right].
    \end{aligned}
  \end{equation*}
\end{definition}

The conditional version $f_k$ describes the fluctuations of $f$ in its $k$-th component, given the other variables $(x_i)_{i \neq k}$.
Some properties of the conditional version of $f$ are extracted from \citep{maurer2021concentration} to improve its introduction.
Note that $f_k(\bX) : \bx \in \Xsp^n \mapsto f_k(\bX)(\bx) $ is a random function that does not depend on the $k$-th component of $\bx$.
Considering a norm $\| \cdot \|_{\psi}$ on random variables, the function $\| f_k(\bX) \|_{\psi} : \Xsp^n \to \R_+$ is defined by  $\| f_k(\bX) \|_{\psi}(\bx) = \| f_k(\bX)(\bx)\|_{\psi}$. 
Remark that $\| f_k(\bX)\|_{\psi}(\bX)$ is a random variable and that $f_k(\bX)(\bX) = f(\bX) - \E[f(\bX) | X_1,\ldots,X_{k-1},X_{k+1},\ldots X_n]$.
The quantity $ \|\| f_k(\bX) \|_{\psi} \|_{\infty}$ refers to its essential supremum. If $\bX$, $\bY$ are independent copies then $\| f_k (\bX) \|_{\psi} = \| f_k (\bY) \|_{\psi}$ and $\| f_k (\bX) \|_{\psi}(\bX)$ is independent and identically distributed to $\| f_k (\bX) \|_{\psi}(\bY)$.

The next result provides a McDiarmid-type inequality for function with sub-exponential conditional versions.

\begin{theorem}[{McDiarmid's inequality \cite[Theorem 4]{maurer2021concentration}}] \label{thm:bernstein_mcdiarmid}
  Let $f : \Xsp^n \to \R$ and $\bX \in \Xsp^n$ a random vector with independent components. Then, for any $\delta > 0$,
  \begin{equation} \label{eq:sub_gamma_macdiarmid}
    f(\bX) - \E f \leq e \left( 2\sigma \sqrt{\log\left(\frac{1}{\delta}\right)} + S \log\left(\frac{1}{\delta}\right) \right),
  \end{equation}
  with probability at least $1 - \delta$, with $\sigma^2 = \left\| \sum_k \| f_k(\bX) \|^2_{\psi_1}\right\|_{\infty}$ and $S = \max_k \left\| \| f_k(\bX) \|_{\psi_1} \right\|_{\infty}$.
\end{theorem}

\begin{proof}
  Since this result is not exactly stated in this form, a short proof is provided in the Appendix \ref{appendix:add_proof}.
\end{proof}

The quantities $f_k$, $\sigma^2$ and $S$ are the substitutes to the differences, variance and scale, respectively, for the standard McDiarmid's inequality. 

\subsection{Proof of Theorem \ref{thm:mcdiarmid_combes}} 

The Theorem \ref{thm:mcdiarmid_combes} can now be proved using Theorem \ref{thm:bernstein_mcdiarmid}.

\begin{proof}
  Let $m = \E[f(\bX) \,|\, \bX \in \Bc]$ and $t > 0$. By the total law of probability,
  \begin{equation*}
    \P\left( f(\bX) - m > t\right) \leq \P(f(\bX) - m > t, \bX \in \Bc) + \P(\bX \notin \Bc).
  \end{equation*}
  The inequality of Theorem \ref{thm:bernstein_mcdiarmid} cannot be directly applied to $\P(f(\bX) - m, \bX \in \Bc)$ since the components of $\bX$ are no longer independent conditionally on $\bX \in \Bc$.

Similarly to \citep{combes2015extension}, the main idea of the proof is to define an extension of $f$ to the whole domain that preserves its conditional versions. 
Let $B : \Xsp^n \times \Xsp^n \to \R_+$ be the pseudometric defined by
  \begin{equation*} 
  B(\bx,\by) = \sum_{i=1}^n b(x_i,y_i),
  \end{equation*}
  for all $\bx,\by \in \Xsp^n$. 
  The extension $\bar{f} : \Xsp^n \to \R$ of $f$ will be defined by
  \begin{equation*}
  \bar{f}(\bx) =  \displaystyle \inf_{\by \in \Bc} f(\by) + B(\bx,\by),
  \end{equation*}
  for all $\bx \in \Xsp^n$.
This extension was also used in the Kirszbraun's theorem \citep{kirszbraun1934zusammenziehende,hiriart1980extension} to extend the range of Lipschitz functions.

\begin{lemma} 
    The extension function $\bar{f}$ satisfies 
  \begin{itemize}
    \item $\bar{f}(\bx) = f(\bx)$ for all $\bx \in \Bc$,
    \item $\displaystyle | \bar{f}(\bx) - \bar{f}(\by) | \leq B(\bx, \by)$ for all $\bx, \by \in \Xsp^n$.
  \end{itemize}
  \end{lemma}

  The proof of the theorem continues using this extension. 
  Let $M = \E[\bar{f}(\bX)]$. By definition of $\bar{f}$,
  \begin{equation*}
  \P(f(\bX) - m > t, \bX \in \Bc) \leq  \P(\bar{f}(\bX) - m > t) = \P\left(\bar{f}(\bX) - M > t + m - M\right).
  \end{equation*}
  
  It remains to find a bound on $M - m$. The expectation $M$ can be decomposed as follows
  \begin{equation*}
    M = \E[\bar{f}(\bX) \indic_{\bX \in \Bc}] + \E[\bar{f}(\bX) \indic_{\bX \notin \Bc}] \leq \P(\bX \in \Bc) m + \E[\bar{f}(\bX) \indic_{\bX \notin \Bc}].
  \end{equation*}
  By definition of $\bar{f}(\bx) = \inf_{\by \in \Bc} f(\by) + B(\bx,\by) \leq \E[ f(\bY) + B(\bx,\bY) \, | \, \bY \in \Bc]$, where $\bY \in \Xsp^n$ is independent of $\bX$, so that
  \begin{equation*}
    \E[\bar{f}(\bX) \indic_{\bX \notin \Bc}] \leq m \P(\bX \notin \Bc) + \E\left[ \E[B(\bX,\bY) \, | \, \bY \in \Bc] \indic_{\bX \notin \Bc} \right].
  \end{equation*}
  Therefore,
  \begin{equation*}
    M - m \leq \frac{1}{\P(\bY \in \Bc)} \E\left[B(\bX,\bY) \indic_{\bX \notin \Bc, \bY \in \Bc}\right],
  \end{equation*}
  where the order of integration on $B$ has been interchanged by Tonelli's theorem. 
  An application of Cauchy-Schwartz inequality gives
  \begin{equation*}
    \begin{split}
      \E\left[B(\bX,\bY) \indic_{\bX \notin \Bc, \bY \in \Bc}\right] & \leq \| B \|_{L^2} \P\left( \bX \notin \Bc, \bY \in \Bc \right)^{1/2}, \\
      & \leq 2 n \| b \|_{\psi_1} \P\left( \bX \notin \Bc, \bY \in \Bc \right)^{1/2}.
      \end{split}
    \end{equation*}
  The independence of $\bX$ and $\bY$ leads to 
  \begin{equation*}
    M - m \leq 2 n \| b \|_{\psi_1} \sqrt{\frac{p}{1-p}}.
  \end{equation*}
  In the regime where $p \leq 3/4$, it holds that $p^{\frac{1}{2}} (1-p)^{-\frac{1}{2}} \leq 2\sqrt{p}$, leading to 
  \begin{equation*}
    M - m \leq 4 n \|b\|_{\psi_1} \sqrt{p},
  \end{equation*}
  which implies
  \begin{equation*}
    \P\left(\bar{f}(\bX) - M > t + m - M\right) \leq \P\left(\bar{f}(\bX) - M > t - 4 n \|b\|_{\psi_1} \sqrt{p} \right).
  \end{equation*}

  To conclude the Theorem \ref{thm:bernstein_mcdiarmid} is used on $\bar{f}$ since
  \begin{equation*}
    \begin{aligned}
    \bar{f}_k(\bX)(\bx) &= \E\left[ \bar{f}(x_1,\ldots,x_{k-1},X_k,x_{k+1},\ldots,x_n) - \bar{f}(x_1,\ldots,x_{k-1},X_k',x_{k+1},\ldots,x_n) | X_k \right] \\
    & \leq \E\left[ b(X_k,X_k') | X_k \right],
    \end{aligned}
  \end{equation*}
  so that $\| \bar{f}_k(\bX)(\bx) \|_{\psi_1} \leq \| \E\left[ b(X_k,X_k') | X_k \right] \|_{\psi_1} \leq \| b \|_{\psi_1}$ by \cite[Lemma 6]{maurer2021concentration}.
  Therefore,
  \begin{equation*}
  \bar{f}(\bX) - m \leq 4 n \|b\|_{\psi_1} \sqrt{p} + e\left( 2\sigma \sqrt{\log\left(\frac{1}{\delta}\right)} + M\log\left(\frac{1}{\delta}\right) \right),
  \end{equation*}
  with probability at least $1-p-\delta$ where $\sigma^2 = n \| b \|_{\psi_1}^{2} $ and $S = \| b \|_{\psi_1}$.
\end{proof}

\begin{remark}
  In Theorem \ref{thm:mcdiarmid_combes}, the function $b$ has to be a pseudometric while Theorem \ref{thm:bernstein_mcdiarmid} does not assume any further structure on $b$.
  It is not clear whether Theorem \ref{thm:mcdiarmid_combes} holds dropping the pseudometric assumption.
\end{remark}

\subsection{Proof of Lemma \ref{lem:bound_conditional_expectation}} \label{sec:proof_main_result:conditional_expectation}

The goal of this section is to prove the Lemma \ref{lem:bound_conditional_expectation} providing the bound on the conditional expectation $\E[\vartheta(\hSc;\Sc)^\beta \, | \, \bX \in \Ac]$.

To start this section, it is important to describe some properties of $\Ac$.
First, the subset $\Ac$, as the empirical risk $\hJJ$, is invariant under reordering of the components of the vector $\bX$. This property will be used several times in the proof.
The section of $\Ac$ is defined for $\bX \in \Ac$ by $I(\bX) = \{ x \in \Xsp \, | \, (x, X_2, \ldots, X_n) \in \Ac \}$. The subset $\Ac$ and its section satisfy the following properties
\begin{equation} \label{eq:prop_Ac_I}
\begin{split}
 \{ \bX \in \Ac \} \cap \{ \bX' \in \Ac \} &= \{\bX \in \Ac \} \cap \{ Y \in I(\bX)\}, \\
 \{ \bX \in \Ac \} \cap \{ \bX' \notin \Ac \} &= \{\bX \in \Ac \} \cap \{ Y \notin I(\bX)\}, 
\end{split}
\end{equation}
recalling that $\bX' = (Y,X_2,\ldots,X_n)$ where $Y$ is an independent copy of $X_1$. 
The events $\bX \in \Ac$ and $\bX' \notin \Ac$ are negatively correlated, meaning that $\P(\bX \in \Ac, \bX' \notin \Ac) \leq \P(\bX \in \Ac) \P(\bX' \notin \Ac)$, see Lemma \ref{lem:negative_correlation}. This property quantifies the intuitive idea that, once $\bX$ fall in $\Ac$ with large probability, the probability that the set of points $\bX'$ (that is $\bX$ with one sample point swapped) fall outside $\Ac$ is small.

\begin{proof}[{Proof of Lemma \ref{lem:bound_conditional_expectation}}]
  The proof starts with the usual decomposition for any $\hphi \in \hSc$
  \begin{equation*}
    \begin{aligned}
      \tau \vartheta( \hphi, \Sc )^{\beta} & \leq \JJ(\hphi) - \JJ_*  & (\text{using Assumption \ref{ass:local_holder_error} and } \bX \in \Ac) \\
      & = \E_Y[l(\hphi,Y)] + \RR(\hphi) + \hJJ_* - \hJJ_* - \JJ_* \\
      & = \E_Y[l(\hphi,Y)] - \hff(\hphi) + \hJJ_* - \JJ_* \\
      & = \frac{1}{n} \sum_{i=1}^n \E_Y[l(\hphi,Y) - l(\hphi,X_i)] + \hJJ_* - \JJ_*.
    \end{aligned}
  \end{equation*}

  Remark that, 
  \begin{equation*}
    \begin{aligned}
        \E\left[ \sup_{\hphi \in \hSc} \frac{1}{n} \sum_{i=1}^n \E_Y[l(\hphi,Y) - l(\hphi,X_i)] \, | \, \bX \in \Ac \right] & \leq \frac{1}{n} \sum_{i=1}^n \E\left[ \sup_{\hphi \in \hSc}  \E_Y[l(\hphi,Y) - l(\hphi,X_i)]  \, | \, \bX \in \Ac \right] \\
        &= \E\left[ \sup_{\hphi \in \hSc}  \E_Y[l(\hphi,Y) - l(\hphi,X_1)]  \, | \, \bX \in \Ac \right],
    \end{aligned}
  \end{equation*}
since $\Ac$ is invariant under permutations of the components,

  The standard argument of \eqref{eq:estimation_error_decomposition_stability} cannot be directly applied here since $Y$ and $X_1$ do not share the same distribution anymore due to the conditioning on $\bX \in \Ac$. 
  It is however possible to split the expectation on $Y$ into the domains (i) $Y \in I(\bX)$ where $Y$ will be swappable with $X_1$ and (ii) $Y \notin I(\bX)$ where $Y$ and $X_1$ do not share the same distribution. Fortunately, the event $Y \notin I(\bX)$ will have a small measure. 
  This leads to the following decomposition
  \begin{equation*}
    \begin{aligned}
      \tau \E\left[ \vartheta(\hSc;\Sc)^{\beta}  \, | \, \bX \in \Ac \right] & \leq E_1 + E_2, \quad \text{with } E_i = \E\left[\sup_{\hphi \in \hSc} e_i \, | \, \bX \in \Ac\right], \\
      \text{and} \quad  e_1 = \E_Y&\left[\left( l(\hphi,Y) - l(\hphi,X_1) \right) \indic_{Y \in I(\bX)} \right], \\ 
      e_2 = \E_Y&\left[\left( l(\hphi,Y) - l(\hphi,X_1) \right) \indic_{Y \notin I(\bX)} \right] + \hJJ_* - \JJ_*. 
    \end{aligned}
  \end{equation*}
  
  \begin{description}
    \item[Bounding $E_1$] 
    Since $\sup_{\phi} \E[g(\phi,X)] \leq \E[\sup_\phi g(\phi,X)]$ for any measurable $g : \MM \times \Xsp \to \R$, the term $E_1$ can further be bounded by
    \begin{equation*}
    \begin{aligned}
        E_1 &\leq \E_{\bX} \left[ \E_Y \left[ \sup_{\hphi \in \hSc}  (l(\hphi,Y) - l(\hphi,X_1)) \indic_{Y \in I(\bX)} \right] \, | \,  \bX \in \Ac \right].
      \end{aligned}
    \end{equation*}
    The next step of the proof consists in bringing out the stability of the minimizers of the empirical risk in $E_1$. 
Let $\tff$, $\tJJ$ and $\tSc$ be the empirical risks and the minimizers associated to $\bX' = (Y,X_1, \ldots, X_n)$. 
Remarking that,
\begin{equation*}
l(\hphi,Y)- l(\hphi,X_1) = n (\tff(\hphi) - \hff(\hphi)) = n (\tJJ(\hphi) - \hJJ(\hphi)) = n (\tJJ(\hphi) - \hJJ_*),
\end{equation*}
since the terms in $\RR$ cancel out, and using the fact that $\bX$ and $\bX'$ have the same distribution gives
\begin{equation*}
\begin{aligned}
  \E_{\bX} \left[ \E_Y \left[\tJJ_* \indic_{Y \in I(\bX)} \right] \indic_{\bX \in \Ac}\right] =  \E_{\bX} \left[ \E_Y \left[ \hJJ_* \indic_{Y \in I(\bX)} \right] \indic_{\bX \in \Ac}\right]. && \text{(see Lemma \ref{lem:expectation_empirical_risk_equal})}
\end{aligned}
\end{equation*}
This leads to the bound
\begin{equation*}
\begin{aligned}
    E_1 & \leq \frac{n}{1-p_n} \E_{\bX,Y} \left[  \frac{1}{n} a(Y,X_1) \vartheta(\hSc ; \tSc)^\alpha  \indic_{\bX \in \Ac,  \bX' \in \Ac} \right],
  \end{aligned}
\end{equation*}
with an argument similar to Lemma \ref{lem:hJJ_error_bound}. Note that the order of integration has been interchanged by Tonelli's theorem. Therefore, using Lemma \ref{lem:hJJ_error_bound} shows that
\begin{equation} 
    \begin{aligned}
    E_1 & \leq \frac{1}{1-p_n}\E\left[ a(Y,X_1) \left(2 \tau^{-1} n^{-1} a(Y,X_1)\right)^{\frac{\alpha}{\beta-\alpha}} \indic_{\bX \in \Ac, \bX' \in \Ac} \right]  \\
    & \leq \frac{(2 \tau^{-1} n^{-1})^{\frac{\alpha}{\beta-\alpha}}}{1-p_n} \left\| a^{\frac{\beta}{\beta-\alpha}} \right\|_{L^2} \sqrt{ \P(\bX \in \Ac, \bX' \in \Ac)}, \label{eq:proof_E1_1}
  \end{aligned}
\end{equation}
using Cauchy-Schwartz inequality. By definition of $p_n$,
    \begin{equation}\label{eq:proof_E1_2}
    \frac{\sqrt{\P(\bX \in \Ac, \bX' \in \Ac)}}{1-p_n} \leq \frac{1}{\sqrt{1-p_n}} \leq 2,
    \end{equation}
    in the regime where $p_n \leq 3/4$. 
    Furthermore, since $\frac{\beta}{\beta-\alpha} \geq 1$
    \begin{equation}\label{eq:proof_E1_3}
      \left\| a^{\frac{\beta}{\beta-\alpha}} \right\|_{L^2} = \left\| a \right\|_{L^{\frac{2\beta}{\beta-\alpha}}}^{{\frac{\beta}{\beta-\alpha}}} \leq \left(\frac{2\beta}{\beta-\alpha} \| a \|_{\psi_1}\right)^{\frac{\beta}{\beta-\alpha}}.
    \end{equation}
    Combining \eqref{eq:proof_E1_1}, \eqref{eq:proof_E1_2} and \eqref{eq:proof_E1_3} leads to
     \begin{equation*}
      \begin{aligned}
          E_1 & \leq 2 (2 \tau^{-1} n^{-1})^{\frac{\alpha}{\beta-\alpha}} \left(\frac{2\beta}{\beta-\alpha} \| a \|_{\psi_1}\right)^{\frac{\beta}{\beta-\alpha}}.
        \end{aligned}
      \end{equation*}

    \item[Bounding $E_2$]
    The second term emerges from the domain where $Y$ and $X_1$ do not have the same distribution. It can be bounded by
   \begin{equation*}
      E_2 \leq 2^{\max(0,\alpha-1)+2} \tau^{-\frac{\alpha}{\beta-\alpha}} \|a\|_{\psi_1}^{\frac{\beta}{\beta-\alpha}} \sqrt{p_n} \left(\left(\tau \|a\|_{\psi_1}^{-1}\right)^{\frac{\alpha}{\beta-\alpha}} \diam(\Sc)^\alpha + c_2(\alpha,\beta) \right),
    \end{equation*}
    using Lemma \ref{lem:bound_E2}, with $c_2(\alpha,\beta) = 2  \max\left( \frac{4\alpha}{\beta - \alpha},1 \right)^\frac{\alpha}{\beta-\alpha}$.
  \end{description}
  Combining the bounds on $E_1$ and $E_2$ gives the desired result.
\end{proof}

The bound on $E_2$ makes the diameter of $\Sc$ appear, requiring the set of minimizers to be bounded. This is not too restrictive in practice, but it is not clear if it is an artifact of the proof or a price to be paid when splitting the domain of integration of $Y$

\section{Applications} \label{sec:applications}

In this paragraph, the main results of Section \ref{sec:main_results} are applied on the examples presented in the introduction.

\subsection{Barycenters in Hadamard Spaces} \label{sec:barycenter_hadamard}

First introduced by \cite{frechet1948elements}, Fréchet means as defined in \eqref{eq:frechet_means} have found numerous applications in geometry \citep{sturm2003probability,villani2003topics}, statistics \citep{pennec2006intrinsic,pelletier2005informative}, data science \citep{beg2005computing,peyre2019computational,bigot2019data} and more.
Note that the behavior of the empirical counterpart has been studied in specific scenarii for example in Riemannian manifold \citep{bhattacharya2003large}, in Wasserstein spaces \citep{ahidar2020convergence,le2022fast} or recently in non-positively curved metric spaces \citep{brunel2023concentration}.
This section is dedicated to Fréchet means defined on Hadamard spaces.

Let $(\Msp, \vartheta) = (\Xsp, \rho)$ a Hadamard space defined by  
\begin{definition} A Hadamard space is a complete geodesic metric space $(\Xsp, \rho)$
fulfilling
    \begin{equation} \label{eq:charact_hadamard}
      \rho(\phi,x)^2 - \rho(\psi,x)^2 - \rho(\phi,y)^2 + \rho(\psi,y)^2 \leq 2 \rho(x,y) \rho(\phi,\psi),
    \end{equation}
for all $x,y,\phi,\psi \in \Xsp$.
\end{definition}
The property \eqref{eq:charact_hadamard} characterizes $\textrm{CAT}(0)$-spaces \cite[Corollary 3]{berg2008quasilinearization} \cite[Proposition 2.4]{sturm2003probability}, that is geodesic metric spaces with non-positive curvature in the sense of Alexandrov.
The class of Hadamard spaces describes a large variety of spaces such as Euclidean spaces, Hilbert spaces, complete metric trees, complete simply-connected Riemannian manifolds with non-positive sectional curvature and more. An introduction to Hadamard spaces was proposed by \citet{bacak2014computing}.
 
The computation of the barycenter of $\mu$ consists in minimizing $\JJ(\phi) = \int_{\Xsp} \rho^2(\phi,x) d\mu(x)$ which corresponds to $l(\phi,x) = \rho(\phi,x)^2$ in \eqref{eq:opt_J}.

It is known that the barycenter of a measure on a Hadamard space exists and is unique provided that there exists a $\phi \in \Xsp$ such that $\E_{X \sim \mu}[\rho(\phi,X)^2] < +\infty$ \cite[Proposition 4.3]{sturm2003probability}.
Furthermore, \citet[Proposition 4.4]{sturm2003probability} showed that Assumption \ref{ass:local_holder_error} holds with $\JJ_0 = +\infty$, $\tau = 1$ and $\beta = 2$.

Assuming that $\| \rho \|_{\psi_1}$ is finite implies in particular that both moments $\E_{X \sim \mu}[\rho(\phi,X)^2]$ and $\E_{X \sim \mu_n}[\rho(\phi,X)^2]$ are finite for some $\phi \in \Xsp$ so that all the assumptions are satisfied as detailed by the next lemma.

\begin{lemma} \label{lem:assumption_hadamard}
  Assuming that $\|\rho\|_{\psi_1}< +\infty$, then this example satisfies Assumptions \labelcref{ass:existstence_minimizer,ass:local_holder_error,ass:convergence_bassin,ass:lojasewicz_modulus,ass:holder_l} with 
  \begin{center}
  \begin{tabular}{c c c c c c c c c}
  \toprule
  $\JJ_0$ & $\beta$ & $\tau$ & $\alpha$ & $ a $ & $\|a\|_{\psi_1}$ & $\eta$ & $\kappa$ & $\iota$ \\ \midrule
  $+\infty$ & 2 & 1 & 1 & $2\rho$ & $2\|\rho\|_{\psi_1}$ & 0 & 0 & 0\\
  \bottomrule
\end{tabular}
\end{center}
\end{lemma}

\begin{corollary} \label{cor:bound_hadamard}
  Assuming that $\|\rho\|_{\psi_1}< +\infty$ then, for any $\delta > 0$
  \begin{equation*}
    \rho(\phi_*,\widehat{\phi}) \leq 16\|\rho\|_{\psi_1} \left( n^{-\frac{1}{2}} + 2\sqrt{\frac{1}{n} \log\left(\frac{1}{\delta}\right)} + \frac{1}{n}\log\left(\frac{1}{\delta}\right) \right),
  \end{equation*}
  with probability at least $1-\delta$ and 
  \begin{equation*}
    \E\left[ \rho(\phi_*,\widehat{\phi})^2 \right] \leq 2^8 \| \rho \|_{\psi_1}^2 n^{-1}.
  \end{equation*} 
\end{corollary}

Similar results are obtained by \citet{le2022fast} and \citet{brunel2023concentration}.

\begin{remark}
In the Euclidean case, the quantity $\| \rho \|_{\psi_1} $ can be bounded by
\begin{equation*}
\| \rho \|_{\psi_1} \leq 2 \| \|X \| \|_{\psi_1} \leq 2 \sqrt{d} \|X\|_{\psi_1},
\end{equation*}
(see Preliminaries) so that the condition $\|\rho\|_{\psi_1} < + \infty$ is met for sub-exponential random variables.
In Hilbert spaces, this condition is however more subtle to grasp. Some examples of random functions fulfilling this condition are given by \citet[Proposition 5.3]{lei2020convergence}
\end{remark}

\begin{remark}
When $\Msp = \Xsp$ is further restricted to be an Euclidean space $\Msp = \Xsp = \R^d$ and endowed with the usual $\ell_2$ norm, the bound of Corollary \ref{cor:bound_hadamard} can be compared with straight calculations.
In this context, the unique barycenter is $\phi_* = \E_\mu[X]$ (in the Pettis integral sense) and the empirical one $\widehat{\phi} = (1/n)\sum_{i=1}^n X_i$.
The concentration property of sub-exponential variables gives that for any $\delta > 0$
\begin{equation*}
  \| \phi_* - \widehat{\phi}\|_2 \leq 2 \sqrt{d} \|X\|_{\psi_1} \left( \frac{1}{n}\log\left( \frac{2}{\delta}\right)+ \sqrt{\frac{2}{n} \log\left( \frac{2}{\delta}\right)} \right),
\end{equation*}
holds with high probability $1-\delta$.
Moreover, further computations show that
\begin{equation} \label{eq:barycenter_euclidean}
  \E \left[\| \phi_* - \widehat{\phi}\|_2^2 \right]= \frac{1}{n} \E \left[ \| X - \E[X] \|_2^2 \right] \leq \frac{16}{n} \| \|X\| \|_{\psi_1}^2 = \frac{16d}{n} \| X \|_{\psi_1}^2.
\end{equation}
The concentration in probability and expectation obtained via Corollary \ref{cor:bound_hadamard} are therefore optimal up to universal constants.
\end{remark}

\subsection{Largest Eigenvector of Covariance Matrices} \label{sec:application_eigenvector}

Principal components analysis (PCA) \citep{pearson1901liii,hotelling1933analysis} is a dimension reduction technique. 
It consists in computing the eigensystems of the covariance matrix of the data to identify the subspaces that contain the largest variability.
In the majority of high-dimensional analyses, the data is represented as points $(x_i)_{i=1}^n$ sampled from the unknown probability measure $\mu$ in $\R^d$. 
The covariance matrix $\Cov(\mu)$ can thus be estimated only from the samples represented by the discrete measure $\mu_n = \frac{1}{n}\sum_{i=1}^n \delta_{x_i}$.
One important question concerns the stability of the subspaces of $\Cov(\mu_n)$ with respect to those of $\Cov(\mu)$. 
The case of the leading direction of variability, which is the largest eigenvector, will be analyzed in this section.

Let $\R^d$ endowed with the Euclidean distance and equipped with a centered probability measure $\mu$.
The covariance matrix of $\mu$ is defined as $\Cov(\mu) = \E_{X \sim \mu} \left[ X X^T\right]$.

Let $\lambda_1 \geq \ldots \geq \lambda_m $ be the eigenvalues of $\Cov(\mu)$ and $\phi_*$ one of the two eigenvectors associated to the largest eigenvalue.
By the Courant–Fischer min-max principle, $\phi_*$ is defined by
\begin{equation}\label{eq:min-max_largest}
  \phi_* \in \argmax_{\phi \in \mathbb{S}^{d-1}} \langle \Cov(\mu) \phi, \phi \rangle.
\end{equation}
This problem fits in the formalism of \eqref{eq:opt_J} with $\Msp = \mathbb{S}^{d-1}$ endowed with the great circle distance $\vartheta(\phi,\psi) = \arccos \langle \phi, \psi\rangle_{\R^d}$, $\Xsp = \R^d$ and $l(\phi, x) = - \langle \phi, x\rangle^2$.
In this context, the set of minimizers of $\JJ$ is $\Sc = \{ \pm \phi_* \}$ and $\vartheta(\phi,\Sc) = \arccos |\langle \phi, \phi_* \rangle|$.

\begin{lemma} \label{lem:assumption_least_eigen} Assuming $\|X\|_{\psi_2} < +\infty$ then this example satisfies Assumptions \labelcref{ass:existstence_minimizer,ass:local_holder_error,ass:convergence_bassin,ass:lojasewicz_modulus,ass:holder_l} with 
\begin{center}
  \begin{tabular}{c c c c c c c c c}
  \toprule
  $\JJ_0$ & $\beta$ & $\tau$ & $\alpha$ & $ a $ & $\|a\|_{\psi_1}$ & $\eta$ & $\kappa$ & $\iota$ \\ \midrule
  $+\infty$ & 2 & $\frac{4}{\pi^2}(\lambda_{1} - \lambda_{2})$ & 1 & $a$ & $4d\|X\|_{\psi_2}^2$ & $0$ & $2 d \exp(-cn)$ & 0\\
  \bottomrule
\end{tabular}
\end{center}
with
\begin{equation*}
  \begin{aligned}
    a(x,y) &= 2\left( \| x \|_2^2 + \| y \|_2^2 \right), \\
    c &= \frac{(\lambda_1-\lambda_2)^2}{32 e^2 (d+1)^2 \|Y\|_{\psi_2}^4 + 8e(d+1)\|Y\|_{\psi_2}^2(\lambda_1-\lambda_2)}.
  \end{aligned}
\end{equation*}
\end{lemma}
\begin{proof}
  See Appendix \ref{sec:proof_eigenvector}.
\end{proof}

Furthermore since $\diam(\Msp) = \pi$, the following bounds hold.
\begin{corollary} \label{cor:least_eigen}
 Assuming $\| X \|_{\psi_2} < + \infty$ then, for any $\delta >0$
\begin{equation*}
\vartheta(\hSc ; \Sc) \leq 8 L \left(n^{-\frac{1}{2}} + 2\sqrt{\frac{1}{n} \log\left(\frac{1}{\delta}\right)}+\frac{1}{n}\log\left(\frac{1}{\delta}\right)\right) + C_1 d^{\frac{1}{4}}e^{-\frac{c}{4}n},
\end{equation*}
with probability at least $1 - e^{-cn} - \delta$ and
\begin{equation*}
  \E\left[ \vartheta(\hSc ; \Sc)^2 \right] \leq 2^6 L^2 n^{-1} + C_2 d^{\frac{1}{2}} e^{-\frac{c}{2} n},
\end{equation*}
with the constant $c$ defined in Lemma \ref{lem:assumption_least_eigen} and where
\begin{equation*}
  \begin{aligned}
    L = \frac{2d \pi^2 \|X\|_{\psi_2}^2}{(\lambda_1-\lambda_2)}, &\quad C_1 = 14 L + 2\sqrt{L \pi},& \quad C_2 = 4(L\pi + 8L^2) + \pi^2.
  \end{aligned}
\end{equation*}
\end{corollary}

\begin{remark}
  The usual way to obtain bounds between $\phi_*$ and $\widehat{\phi}$ is to use matrix perturbation tools such as the Davis-Kahan Theorem \cite[Theorem 2]{yu2015useful} which gives
\begin{equation*}
\inf_{\phi_* \in \Sc}\| \widehat{\phi} - \phi_* \|_2 \leq \sqrt{2} \sin(\arccos |\langle \phi_*,\widehat{\phi}\rangle|) \leq 2 \sqrt{2} \frac{\| \Cov(\mu_n) - \Cov(\mu)\|_{2\to 2}}{\lambda_{1} - \lambda_2}.
\end{equation*}
The deviation of $\Cov(\mu_n)$ from $\Cov(\mu)$ can then be bounded using, for example, the matrix Bernstein inequality of Lemma \ref{lem:sample_covariance_subexp} giving, for all $t > 0$,
\begin{equation*}
\begin{aligned}
  \P\left(\inf_{\phi_* \in \Sc}\| \widehat{\phi} - \phi_* \|_2 > t \right) && \leq 2 d \exp\left(\frac{-nt^2}{2(L^2 + L t) }\right).
\end{aligned}
\end{equation*}
On the other hand, the bound of Corollary \ref{cor:least_eigen} implies that
\begin{equation*}
  \P\left(\inf_{\phi_* \in \Sc}\| \widehat{\phi} - \phi_* \|_2  > t-8Ln^{-\frac{1}{2}}-C_1e^{-\frac{c}{4}n} \right) \leq 2 d \exp(-cn) + \exp\left(\frac{-nt^2}{256L^2 + 16L t}\right).
\end{equation*}
for all $t > 0$, which are similar, up to universal constants.
\end{remark}

\begin{remark}
  This analysis can also be applied to the eigenvector associated to the smallest eigenvalue and replacing $\lambda_1 - \lambda_2$ by $\lambda_{d-1} - \lambda_d$.
  The general framework developed in this paper also allows to cover eigenvector associated to eigenvalues with multiplicity larger than 1.
  The analysis can be extended to a whole subspace of dimension $r \leq d$ by setting $\Msp = V_{r}(\R^d)$ the Stiefel manifold of dimension $r$ in $\R^d$.
\end{remark}

\subsection{LASSO}

The goal of supervised learning is to predict outputs $V \in \Vsp \subseteq \R$ from inputs $Y \in \Ysp \subseteq \R^d$ given the knowledge of $n$ pairs $x_i = (y_i,v_i) \in \Ysp \times \Vsp$. 
In the usual machine learning framework, these pairs are assumed to be independently sampled from a probability distribution $\mu$ on $\Xsp = \Ysp \times \Vsp$.

The marginal $\mu_\Ysp$, of $\mu$ along $\Ysp$, models the randomness in the inputs.
A common model for the randomness in the outputs is to assume that there exists a measurable function $g_* : \Ysp \to \Vsp$ such that $V = g_*(Y) + \epsilon$ where $\epsilon$ is the noise of the model represented by the marginal $\mu_\Vsp$. For instance, the noiseless case $\epsilon = 0$ is obtained when $\mu_\Vsp$ is the Dirac in $g_*(X)$ while the additive Gaussian noise one is obtained when $\mu$ is the convolution between $\mu_\Ysp$ and the centered Gaussian distribution.

The best predictor, in the least-square sense, is defined as the minimizer of
\begin{equation} \label{eq:MSE}
  g_0 \in \argmin_{g : \Ysp \to \Vsp} \E_\mu[(V - g(Y))^2],
\end{equation}
where $g$ is picked in the set of measurable functions from $\Ysp \to \Vsp$. When $Y$ and $V$ are square-integrable, the optimum $g_0$, also called the Bayes predictor, is defined by $g_0(Y) = \E[V \,|\, Y]$.
Even with this closed form, this problem is intractable for a computer due to the infinite dimension of the space of measurable functions from $\Ysp \to \Vsp$. 

This issue is usually circumvented by parameterizing the space of solutions as $\mathcal{H} = \{ g \, | \, g(y) = \langle \phi, \theta(y) \rangle, \phi \in \R^m \}$ where $\phi \in \R^m$ and $\theta : \Ysp \to \R^m$ are the feature vectors and solving 
\begin{equation} \label{eq:MSE_parameterized}
  \phi_0 \in \argmin_{\phi \in \R^m} \ff(\phi) = \E_\mu[(V - \langle \theta(Y), \phi \rangle)^2]
\end{equation}
in place of \eqref{eq:MSE}. The class $\mathcal{H}$ has to be large enough to approximate well the Bayes predictor. This question is fundamental but out of the scope of this paper \citep[see][]{pinkusNwidthsApproximationTheory1985,tsybakovIntroductionNonparametricEstimation2009}.
Still, the problem \eqref{eq:MSE_parameterized} cannot be solved since the joint distribution $\mu$ is not known and only samples drawn from it are accessible. 
A direct minimization of $\hff$, the empirical version of \eqref{eq:MSE_parameterized}, might however result in overfitting, so that the penalized empirical minimization problem is solved
\begin{equation} \label{eq:empirical_lasso}
  \hphi \in \argmin_{\phi \in \R^m} \hJJ(\phi) = \frac{1}{2 n} \|\Theta \phi - V \|_2^2 + \lambda \| \phi \|_1,
\end{equation}
 where $V \in \R^n$ being the vector of outputs $V = (v_i)_{i=1}^n$ and $\Theta \in \R^{n \times m}$ is the design matrix with the vectors $(\theta(y_i)^T)_{i=1}^n$ as rows.
The problem $\eqref{eq:empirical_lasso}$ is referred to as LASSO (Least Absolute Shrinkage and Selection Operator) by \citet{tibshirani1996regression}. The $\ell_1$ regularization further promotes the sparsity of the minimizers $\hphi$ of $\hJJ$. It has been thoroughly studied in relation to the problem of sparse signal recovery \citep[see for example][for a comprehensive treatment and bibliographic details]{candes2006stable,donoho2006most,wainwright2019high}.  
In particular, a fundamental question in sparse recovery related to the content of this paper is to understand the convergence of the minimizers of the empirical LASSO \eqref{eq:empirical_lasso} towards those of the approximated Bayes predictor \eqref{eq:MSE_parameterized}. High probability recovery guarantees are for example obtained assuming that $\phi_0$ is sparse and that the design matrix $\Theta$ satisfies the so-called restricted eigenvalue condition \citep{raskutti2010restricted,rudelson2012reconstruction,wainwright2019high}. This condition imply a quadratic growth of $\hff$ around its minimizers in directions that are ``sufficiently sparse''.
\citet{koltchinskii2009sparsity} carried another analysis on the excess risk $\ff(\hphi) - \ff(\phi_0)$ by introducing the intermediate problem
\begin{equation} \label{eq:exact_lasso}
  \phi_* \in \argmin_{\phi \in \R^m} \JJ(\phi) = \ff(\phi) + \lambda \| \phi \|_1,
\end{equation}
and by separately studying the bias $\ff(\phi_*) - \ff(\phi_0)$ and the random error $\ff(\hphi) - \ff(\phi_*)$. Finally, the excess risk between $\ff(\hphi) - \ff(\phi_0)$ is obtained by balancing both errors. While the bias part is bounded assuming some sparsity of the solution $\phi_0$ and a local strong convexity of $l$, the random error is bounded using empirical process theory. Note that asymptotic convergence between \eqref{eq:empirical_lasso} and \eqref{eq:exact_lasso} has also been studied by \citet{knight2000asymptotics}. 

In this section, the random error between $\hSc$ and $\Sc$ is studied using the formalism of \eqref{eq:opt_J} with $\Xsp = \Ysp \times \Vsp$, $l : \Xsp \to \R_+$ defined by $l(\phi, (y,v)) = \frac{1}{2} (\langle \phi, \theta(y) \rangle - v)^2$, $\Msp = \R^m$ endowed with the Euclidean distance and $\RR = \lambda \| \cdot\|_1$.

The next result shows that all the assumptions required to apply Lemma \ref{lem:bound_conditional_expectation} and Theorem \ref{thm:concentration} are satisfied.
Proving the local error bound for $\JJ$ and $\hJJ$ is the most challenging part. Since the function $\JJ$ and $\hJJ$ are convex piecewise polynomials of degree 2, they both satisfy a local error bound of Assumption \ref{ass:local_holder_error} with $\beta = 2$ \cite[Corollary 9]{bolte2017error} for any $\JJ_0 > \JJ_*$ with an adequate $\tau(\JJ_0)$. This powerful result however does not specify the values of the constants $\JJ_0$ and $\tau$ which are important to derive proper bounds and to check Assumption \ref{ass:lojasewicz_modulus}. 

The following analysis will rely on the polyhedral geometry of the $\ell_1$ ball for which error bounds analyses are available through the Hoffman's lemma \citep{hoffman2003approximate,guler2010foundations}.

\begin{definition}[Hoffman's error bound]
Let $P$ be a polyhedron defined by $P = \{ x \in \R^q : A x \leq a \}$ for some $A \in \R^{p \times q}$ and $a \in \R^p$. Let $O = \{ x \in \R^q : E x = e\}$ for some $E \in \R^{r \times q}$ and $e \in \R^r$.
Assume that $O \cap P \neq \emptyset$. Then there exists a constant $H \geq 0$, depending only on the pair $(O,P)$ such that 
\begin{equation*}
  \inf_{y \in O \cap P} \| x - y \| \leq H \| Ex - e\|, \quad \forall x \in P.
\end{equation*}
\end{definition}
Alternatively, the Hoffman constant of the pair $(O,P)$ can be written as
\begin{equation} \label{eq:def_hoffman}
  H = \max_{\substack{ I \subseteq \{ 1, \ldots, r + p\} \\ C_I \text{ full row rank}}} \frac{1}{\sigma_{|I|}(C_I)},
\end{equation}
where $C_I \in \R^{|I| \times l}$ refers to the submatrix built from the rows of indexes in $I$ of the matrix $C = [E^T, A^T]^T \in \R^{(r+p) \times q}$ \citep{pena2021new}.

The polyhedral geometry of the $\ell_1$ ball is finally exploited using the following result and will be key to prove Assumptions \ref{ass:local_holder_error} and \ref{ass:lojasewicz_modulus}.

\begin{lemma}[{\citet[Lemma 10]{bolte2017error}}] \label{lem:hoffman}
  Let $\LL : \R^q \to \R$ defined by $\LL(\phi) = \frac{1}{2}\|A \phi -b\|_2^2 + \lambda\|\phi\|_1$ where $\lambda >0$, $b \in \R^p$ and $A$ is a matrix of size $p \times q$. 
  Let $R_* > 0$ be chosen such that the minimizers $\phi_*$ of $\LL$ satisfy $\| \phi_* \|_1 \leq R_*$. Fix $R > R_*$, then
  \begin{equation*}
    \tau \vartheta(\phi,S)^2 \leq \LL(\phi) - \LL_*,
  \end{equation*}
  for all $\phi \in \R^q$ such that $\|\phi\|_1 \leq R$ where
  \begin{equation} \label{eq:tau_lasso}
    \tau = \left(4 H^2\left[1 + \lambda R + (R \|A\| + \|b\|)(4 R \|A\| + \|b\|)\right]\right)^{-1},
  \end{equation}
  and $H = H(A,\lambda) > 0$ is the Hoffman's constant \eqref{eq:def_hoffman} of the matrix 
  \begin{equation} \label{eq:hauffman_matrix}
  C(A,\lambda) = 
  \begin{pmatrix}
    E & -1_{\R^{2^q \times 1}}\\
    0_{\R^{1\times q}} & 1 \\
    A & 0_{\R^{p\times 1}} \\
    0_{\R^{1\times q}} & \lambda
  \end{pmatrix} \in \R^{(2^q+1+p+1)\times (q+1)},
  \end{equation}
  where $E$ is the matrix of size $2^q \times q$ with rows being all the possible distinct vectors of the form $(\pm 1,\ldots, \pm 1)$.
\end{lemma}

\begin{lemma}\label{lem:assumption_lasso}
Assuming that $\| \theta(Y) \|_{\psi_2} < + \infty$ and $\| V \|_{\psi_2} < + \infty$ then this example satisfies Assumptions \labelcref{ass:existstence_minimizer,ass:local_holder_error,ass:convergence_bassin,ass:lojasewicz_modulus,ass:holder_l} with 
  \begin{center}
    \begin{tabular}{c c c c c c c c c}
  \toprule
  $\JJ_0$ & $\beta$ & $\tau$ & $\alpha$ & $ a $ & $\|a\|_{\psi_1}$ & $\eta$ & $\kappa$ & $\iota$\\ \midrule
  $\E[V^2]$ & 2 & \eqref{eq:tau_lasso_used} & 1 & \eqref{eq:a_lasso} & \eqref{eq:a_lasso_subexp} & $2 e^{-c_1 n}$ & $2 m e^{-c_2n}$ & $2 e^{-c_1 n}$\\
  \bottomrule
\end{tabular}
\end{center}
where the constants are defined with $A = \E[\theta(Y) \theta(Y)^T]^{\frac{1}{2}} \in \R^{m \times m}$ as 
\begin{equation} \label{eq:tau_lasso_used}
\begin{split}
  \tau & = \left(4 H^2\left[1 + \JJ_0 + \left(\JJ_0 \|A\| \lambda^{-1} + \sqrt{\JJ_0}\right)\left(4 \JJ_0 \|A\| \lambda^{-1} + \sqrt{\JJ_0}\right)\right]\right)^{-1},
\end{split}
\end{equation}
with $H$ being the Hoffman constant of the matrix $C(A,\lambda)$ defined in \eqref{eq:hauffman_matrix} and
\begin{equation} \label{eq:a_lasso_subexp}
  \|a\|_{\psi_1} = 4Rm \|\theta(Y)\|_{\psi_2}^2 + 4 \sqrt{m}\|\theta(Y)\|_{\psi_2} \|V\|_{\psi_2}.
\end{equation}
The constants $c_1$ and $c_2$ are defined by 
\begin{equation*}
  \begin{aligned}
  c_1 &= \frac{\|V\|_2^4}{2(e^2 \| V\|_{\psi_2}^4 + e \| V\|_{\psi_2}^2\|V\|_2^2)}, \\
  c_2 &= \frac{c^2}{2(e^2(m+1)^2\|\theta(Y)\|_{\psi_2}^4 + e (m+1)\|\theta(Y)\|_{\psi_2}^2 c)}, \\
  c &= (\sqrt{2}-1) \min\left(\frac{1}{\sqrt{2} H^2}, \|A\|^2 \right).
  \end{aligned}
\end{equation*}
\end{lemma}
\begin{proof}
See Section \ref{sec:proof_lasso} for the proof and the various expressions.
\end{proof}

\begin{corollary} \label{cor:lasso}
  Assuming that $\| \theta(Y) \|_{\psi_2} < + \infty$ and $\| V \|_{\psi_2} < + \infty$ then, for any $\delta > 0$
  \begin{equation*}
    \vartheta(\hSc ;  \Sc) \leq 8 L \left(n^{-\frac{1}{2}} + 2 \sqrt{\frac{1}{n}\log\left(\frac{1}{\delta}\right)}+\frac{1}{n} \log\left(\frac{1}{\delta}\right)\right) + 14 L p_n^\frac{1}{4},
  \end{equation*}
  with probability at least $1-p_n-\delta$ and
  \begin{equation*}
    \E\left[\vartheta(\hSc ;  \Sc)^2\right] \leq  L^2 2^6 n^{-1} + C_2 p_n^{-\frac{1}{2}},
  \end{equation*}
  where $C_2 = 2^5 + (2\|V\|^2_{\psi_2}+ 2^5\|V\|^4_{\psi_2})\lambda^{-2}$ and $L = \tau^{-1}\|a\|_{\psi_1}$.
\end{corollary}
\begin{proof}
See Appendix \ref{sec:proof_lasso}.
\end{proof}

\citet{koltchinskii2009sparsity} obtained bounds either on the excess risk or on the $\ell_2$ distance between the functions $\langle \theta(\cdot), \phi_* \rangle$ and $\langle \theta(\cdot), \hphi \rangle$. 
The result of Corollary \ref{cor:lasso} shows that concentration bounds on the estimation error between \eqref{eq:exact_lasso} and \eqref{eq:empirical_lasso} in terms of distance between their minimizers are also accessible.
This additional granularity needed a local error bound that is proved via an intricate analysis of the geometry of the LASSO problem. 
The recent result of Lemma \ref{lem:hoffman} was leveraged to obtain a local error bound involving the Hoffman's constant of a matrix defined by the regularization parameter and the covariance matrix of the feature vectors \eqref{eq:hauffman_matrix}.
The most challenging part of Lemma \ref{lem:assumption_lasso} was to prove the Assumption \ref{ass:lojasewicz_modulus}, that is the concentration of $\htau$ towards $\tau$. This step involves the concentration of the Hoffman constant \eqref{eq:def_hoffman} of the matrix $C\left(\left( \frac{1}{n} \sum_{i=1}^n\theta(y_i) \theta(y_i)^T\right)^\frac{1}{2}, \lambda\right)$ towards the one of $C\left( \E\left[ \theta(Y) \theta(Y)^T \right]^\frac{1}{2}, \lambda\right)$ \eqref{eq:hauffman_matrix}.
This technicalities essentially boil down to the concentration of the sample covariance matrix to its exact counterpart which can be established with tools such as Theorem \ref{thm:matrix_bernstein}.

Even though the present result sheds new light on the concentration of the solutions of the empirical LASSO problem, it still need to be connected to the underlying problem of finding bounds on the distance between $\phi_0$ and $\hphi$.
Answering this type of question requires the introduction of hypotheses on the support of the solutions $\phi_0$ and to incorporate this support in the definition of the Hoffman constant.
Furthermore, the dependency of the Hoffman constants of $C(A,\lambda)$ with respect to $\lambda$ has to be properly understood in order to balance the bias and variance terms.
From the author's perspective, there is still a substantial amount of work to accurately address the question of sparse recovery using Corollary \ref{cor:lasso}.

\subsection{Wasserstein Barycenter} \label{sec:wasserstein_barycenter}

Studying barycenters in Hadamard spaces was quite straightforward due to their non-positive curvature that implied Assumptions \ref{ass:existstence_minimizer}, \ref{ass:local_holder_error} (with $\JJ_0 = + \infty$), \ref{ass:lojasewicz_modulus} and \ref{ass:holder_l} (see Section \ref{sec:barycenter_hadamard}). 
The situation is more delicate in spaces of non-negative curvature, and it represents a limitation of the proposed analysis that cannot be applied verbatim. Deeper investigations are thus required to obtain concentration bounds in this case.
A perfect example of such spaces, which has drawn attention in the last decade, is the Wasserstein space. A brief introduction is presented here, but complete treatments can be found in good textbooks \citep[such as][]{villani2003topics,ambrosio2008gradient,ambrosio2013user}.

\subsubsection{The Wasserstein Space}
Given $(H,h)$ a Polish space, we let $\Pc_2(H)$ be the set of all Borel probability measures $\phi$ on $H$ such that $\E_{X \sim \phi} [h(X,y)^2] < +\infty$ for all $y \in H$. Define $\Pi_{\phi,\psi}$, the set of couplings between two measures $\phi$ and $\psi$ in $\Pc_2(H)$, as the set of probability measures $\pi$ on $H \times H$ having $\phi$ and $\psi$ as marginals. The Wasserstein distance between $\phi$ and $\psi$ is defined as 
\begin{equation}\label{eq:wasser_dist}
  W_2^2(\phi,\psi) = \inf_{\pi \in \Pi_{\phi,\psi}} \E_{(X,Y) \sim \pi} [h(X,Y)^2].
\end{equation}

This section focuses on the case where $H$ is a Hilbert space endowed with the distance $h$ associated to the inner product, that is the spaces $(\Msp,\vartheta) = (\Xsp,\rho) = (\Pc_2(H),W_2)$. The Wasserstein space $\Msp$ is known to be geodesic and with non-negative curvature \cite[see][Section 7.3]{ambrosio2008gradient}. Let $\mu$ be a probability measure on $\Pc_2(\Msp)$, its barycenters are defined as the Fréchet means
\begin{equation} \label{eq:wasserstein_barycenter}
  \phi_* \in \argmin_{\phi \in \Pc_2(\Msp)} \JJ(\phi) =  \int W_2^2(\phi,\psi) d\mu(\psi),
\end{equation}
that is setting $l(\phi,\psi) = W_2^2(\phi,\psi)$ in the formalism of \eqref{eq:opt_J}. The existence of the barycenter and its empirical counterpart is guaranteed in most reasonable settings \citep[see][]{agueh2011barycenters,le2017existence}, for example when $(H,h)$ is a separable locally compact geodesic space.

\subsubsection{The Difficulties of Non-Negative Curvature}

As observed by \citet[Theorem 7.3.2]{ambrosio2008gradient}, the barycenter function \eqref{eq:wasserstein_barycenter} and its empirical counterpart are not strongly convex. This is due to the positive curvature of $(\Pc_2(H),W_2)$, and actually $\JJ$ and $\hJJ$ exhibit some concavity property, in the sense of \eqref{eq:concavity_PC}. Therefore, the functions $\JJ$ and $\hJJ$ cannot be strongly convex and  the global error bound of Assumption \ref{ass:local_holder_error} with $\JJ = +\infty$ cannot be obtained via these means. 

The lack of strong convexity of the barycenter functions does not imply that a global error bound of Assumption \ref{ass:local_holder_error} with $\JJ = +\infty$ cannot hold as it represents a weaker notion of strong convexity (see Section \ref{sec:application_eigenvector} for an example of non-convex function with a global error bound). The situation becomes more complex when observing that such a bound cannot hold, for all measures $\mu \in \Pc_2(\Msp)$, since this would imply that $\Msp$ has non-positive curvature \cite[Theorem 4.9]{sturm2003probability}.
Furthermore, such a bound cannot hold uniformly at the vicinity of the barycenters, since the curvature of the balls of $\Msp$ of any radius can be unbounded \cite[Section 2.5]{stromme2020wasserstein}.

Similarly, Assumption \ref{ass:holder_l} cannot hold for all $x,y,\phi,\psi \in \Msp$ since this would imply that $\Msp$ has a non-positive curvature (see Section \ref{sec:barycenter_hadamard}). It would however hold for any $\phi,\psi$ on any bounded subset of $\Msp$.

\subsubsection{When Convergence can be Proved}

After having drawn a catastrophic picture of the situation, proving that the empirical barycenters converge to the exact ones with parametric rate is surprisingly still possible. 

Additional work is however required to characterize measures $\mu \in \Pc_2(\Msp)$ for which convergence can be proved. The following derivation are from \citet{ahidar2020convergence,le2022fast} from which we state the results relevant for the purposes of this section while referring to them for a complete treatment. 

The first ingredient is the fact that the tangent space $\TT_{\phi_*}\Msp$ of $\Msp$ at the barycenter $\phi_*$ can be identified to a Hilbert space \citep[see][Theorem 7]{le2022fast}. The distance induced by the metric will be denoted $\| \cdot \|_{\phi_*}$.
Second they show (in their Corollary 16) that when the measure $\mu \in \Pc_2(\Msp)$ has one barycenter $\phi_*$ such that every $\phi \in \supp(\mu)$ is the pushforward of $\phi_*$ by a the gradient of an $\alpha$-strongly convex and $\beta$-smooth function $\zeta_{\phi_* \to \phi}$ (the Kantorovich potential), that is $\phi = \nabla \zeta_{\phi_* \to \phi} \sharp \phi_*$, the function $\JJ$ satisfies the global error bound of Assumption \ref{ass:local_holder_error} with $\beta = 2$, $\tau = \frac{1}{\beta-\alpha - 1}$ and $\JJ_0 = +\infty$.
Importantly, under this assumption $\tau \vartheta(\hphi,\Sc)^2 \leq \| \frac{1}{n} \log_{\phi_*} X_i \|_{\phi_*}^2$. In other words, the problem of finding a barycenter of $\mu$ can be lifted to finding a barycenter of $\log_{\phi_*} \sharp \mu$ on the tangent space $\TT_{\phi_*} \Msp$, which has a Hilbert structure, from which parametric rates from Section \ref{sec:barycenter_hadamard} can be used.

\subsection{Entropic-Wassertein Barycenters}

Establishing convergence guarantees on the empirical Wasserstein barycenters towards their exact counterpart appeared to be a delicate task (see Section \ref{sec:wasserstein_barycenter}). 
This obstacle was due to the non-negative curvature of $(\Pc_2(\R^d), W_2)$ that lead to an Assumption \ref{ass:local_holder_error} difficult to check for general measures.

To circumvent these difficulties, it was proposed by \citet{bigot2019penalization,kroshnin2017fr} to seek a regularized barycenter by adding an entropic regularization to the Wasserstein variance function.
\citet{bigot2019penalization} proved the existence and uniqueness of the regularized barycenters as well as convergence of their empirical counterpart.
Recently, \cite{carlier2021entropic} derived regularity properties of the entropic-barycenters and established central limit theorems.

Let $\Omega$ be a bounded subset of $\R^d$, of positive Lebesgue measure and endowed with the Euclidean distance. Define $\Msp = \Pc_2^{ac}(\Omega)$ the subset of measures of $\Pc_2(\Omega)$ that are absolute continuous with respect to the Lebesgue measure.
The space $\Msp$ will be endowed with the total variation distance, which coincides up to a factor $\frac{1}{2}$ with the $\ell^1$ norm on $\Pc_2^{ac}(\Omega)$, that is $\vartheta = \frac{1}{2} \| \cdot - \cdot\|_1$. This choice might be surprising at first but will be substantiated later in this section.
Furthermore, let $(\Xsp,\rho) = (\Pc_2(\Omega),W_2)$.
The entropic-barycenter of a probability measure $\mu \in \Pc_2(\Msp)$ is defined, as proposed by \citet{bigot2019penalization}, as the minimizer of 
\begin{equation*}
  \JJ(\phi) = \frac{1}{2} \int_{\Pc_2(\Omega)} W_2^2(\phi,\psi) d\mu(\psi) + \lambda \RR(\phi),
\end{equation*}
where $\lambda > 0 $ is the regularization parameter and $\RR$ is the negative entropy defined for every $\phi \in \Pc_2(\Omega)$ by
\begin{equation*}
  \RR(\phi) = \left\{
  \begin{aligned}
    &\int_\Omega f \log f dx & \quad \text{if } \phi = f dx, \\
    &+ \infty & \quad \text{otherwise}.
  \end{aligned}
  \right.
\end{equation*}
This choice of regularization enforces the Wasserstein barycenter to be absolutely continuous with respect to the Lebesgue measure. The measures $\phi_*$ and $\widehat{\phi}$, as the minimizer of $\JJ$ and $\hJJ$ respectively, therefore admit a density. 
We will slightly abuse notation and use the same one for the measures and their densities.
Any probability measure $\mu$ on $\Msp$ satisfying $\int_{\Pc_2(\Omega)} \E_{X \sim \psi}[X^2] d\mu(\psi) < +\infty$, admits a unique entropic-barycenter \citep[see][Proposition 2.1]{carlier2021entropic}.

While the Wasserstein distance squared is not geodesically convex (see Section \ref{sec:wasserstein_barycenter}), it is convex with respect to the linear structure in $\Pc_2(\Omega)$, that is considering paths like $\phi_t = t \phi_1 + (1-t) \phi_0$ for $\phi_0, \phi_1 \in \Pc_2(\Omega)$ and $t \in [0,1]$, \citep[see Proposition 7.17][]{santambrogio2015optimal}. This leads to convex functions $\ff$ and $\hff$.
Furthermore, the negative entropy is strongly convex, for the linear structure, with respect to the total variation distance on $\Pc_2^{ac}(\Omega)$. These two properties imply the local error bounds for $\JJ$ and $\hJJ$ so that Assumptions \ref{ass:local_holder_error} and \ref{ass:lojasewicz_modulus} hold.

Checking Assumption \ref{ass:holder_l} is more delicate since one has to bound the left hand side of \eqref{eq:holder_l} with $l = \frac{1}{2}W_2^2$ by a quantity involving the total variation which endows $\mathcal{P_2}^{ac}(\Omega)$ with a different topology than with $W_2$.

\begin{lemma} \label{lem:assumption_regbar}
  Assuming that $\Omega$ is bounded and is of positive Lebesgue measure, this example satisfies Assumptions \labelcref{ass:existstence_minimizer,ass:local_holder_error,ass:convergence_bassin,ass:lojasewicz_modulus,ass:holder_l} with
  \begin{center}
  \begin{tabular}{c c c c c c c c c}
  \toprule
  $\JJ_0$ & $\beta$ & $\tau$ & $\alpha$ & $a$ & $\|a\|_{\psi_1}$ & $\eta$ & $\kappa$ & $\iota$ \\ \midrule
  $+\infty$ & 2 & 2 & $1$ & $4\diam(\Omega)^2$ & $4\diam(\Omega)^2$ & $0$ & $0$ & $0$ \\
  \bottomrule
\end{tabular}
\end{center}
\end{lemma}
\begin{proof}
See Appendix \ref{sec:proof_entropic_wass}.
\end{proof}

\begin{corollary}\label{cor:assumption_regbar}
Under the assumption of Lemma \ref{lem:assumption_regbar}, for any $\delta > 0$
\begin{equation*}
  \vartheta(\hphi, \phi_*) \leq 16\diam(\Omega)^2 \left(n^{-\frac{1}{2}} + 2\sqrt{\frac{1}{n} \log\left(\frac{1}{\delta}\right)}+\frac{1}{n}\log\left(\frac{1}{\delta}\right)\right),
\end{equation*}
with probability at least $1-\delta$ and 
\begin{equation*}
  \E\left[ \vartheta(\hphi, \phi_*)^2  \right] \leq 2^8 \diam^4(\Omega) n^{-1}.
\end{equation*}
\end{corollary}

As a comparison, \citet{bigot2019penalization} obtained a bound of kind $\E[\textrm{KL}(\hphi|\phi_*)^2] \leq c n^{-1}$, where $\textrm{KL}$ refers to the Kullback-Leibler divergence, for $d=1$ and some constant $c > 0$. The general case $d \geq 2$ requires additional smoothness assumption.
Expressing this bound in terms of total variation via the Pinsker's inequality would give $\E[\| \hphi - \phi_* \|_1^4] \leq c n^{-1}$.
The bound obtained here in Corollary \ref{cor:assumption_regbar} appears to be faster and does not require additional smoothness assumptions to deal with $d \geq 2$.

\section{Discussion}

The concentration of the empirical minimizers of the risk \eqref{eq:opt_hatJ} to the ones of \eqref{eq:opt_J} is a fundamental question in statistical learning.
Instead of the usual guarantees on the estimation error, concentration inequalities on the distance between the sets of minimizers are investigated.

In particular, this work identifies a set of assumptions, that is rather weak, and should fit a wide variety of estimations problems, with spaces that are unbounded and infinite dimensional and cost functions that are unbounded and non-convex. 
Therefore, these assumptions allow to describe a mechanism that seems to govern the concentration of the empirical minimizers in many estimation problems.
More precisely, there exists a regime $\bX \in \Ac$, happening with high-probability, where the empirical minimizers are stable. This stability was leveraged in Theorem \ref{thm:concentration} to obtain parametric concentration rates. The result was obtained using a new McDiarmid-type inequality for functions with sub-exponential differences on subsets of large measure.
The assumptions were verified on a selection of estimation problems showing their relevance and showcasing the optimality of the bounds of the Theorem \ref{thm:concentration}.

Even tough this work establishes the regime of stability of the empirical minimizers and the machinery to obtain parametric concentration rates from it, the set of assumptions has to be checked for each problem, which is quite unsatisfactory in the quest of user-friendly theorems.
While some clues were provided for the general treatments of Assumptions \ref{ass:convergence_bassin} and \ref{ass:holder_l}, the most interesting and probably the most open ones seems to be the concentration of the \L{}ojasiewicz constant in Assumption \ref{ass:lojasewicz_modulus}.
Additionally to the technical questions raised in the core of the paper, Theorem \ref{thm:concentration} leaves aside the case $\beta = \alpha$ which might be of great interest for applications where the cost function $l$ has, for example, a Lipschitz continuous gradient.

\acks{
The author would like to thank the anonymous reviewers and the editor, Joseph Salmon, for their time and effort in handling the paper. 
Their careful review and constructive comments have greatly improved the quality of this manuscript. 
}

\appendix

\section{Additional Proofs} \label{appendix:add_proof}

\subsection{General Concentration Results}

\begin{proof}[{Proof of Theorem \ref{thm:bernstein_mcdiarmid}}]
  The Theorem 4 of \citet{maurer2021concentration} is not exactly stated in the form of \eqref{eq:sub_gamma_macdiarmid}. More precisely, it states that
  \begin{equation*}
     \P\left( f(\bX) - \E f \geq t \right) \leq \exp\left( - \frac{t^2}{4 e^2 \sigma^2 + 2 e S t }\right),
  \end{equation*}
  for all $t \geq 0$, which is slightly weaker than \eqref{eq:sub_gamma_macdiarmid}.
  However, a closer inspection of the proofs shows that 
  \begin{equation*}
    \log \E\left[ e^{t(f - \E f)}\right] \leq \frac{t^2 e^2 \sigma^2}{1 - t e S},
  \end{equation*}
  for $t \leq (eS^{-1})$. This shows that $f(\bX) \in \mathrm{sub}\Gamma_+(2e^2\sigma^2,eS)$ which implies the tail \eqref{eq:sub_gamma_macdiarmid}, see Section \ref{sec:preliminaries}.
 \end{proof}

\subsection{Supporting Results for the Proof of Lemma \ref{lem:bound_conditional_expectation}}

\begin{lemma} \label{lem:expectation_empirical_risk_equal}
In the setting of the proof of Lemma \ref{lem:bound_conditional_expectation},
  \begin{equation*}
    \E_{\bX} \left[ \E_Y \left[\left(\tJJ_* - \hJJ_*\right) \indic_{Y \in I(\bX)} \right] \indic_{\bX \in \Ac}\right] = 0.
  \end{equation*}
\end{lemma}

\begin{proof}
  This result will be clear, as long as the order of integration can be exchanged, because $\indic_{\bX \in \Ac}$ and $\indic_{\bX' \in \Ac}$ have the same distribution.
  This will be achieved checking that
  \begin{equation*}
    \E_{\bX,Y} \left[ \left|\tJJ_* - \hJJ_*\right| \indic_{\bX \in \Ac, \bX' \in \Ac}\right] < +\infty.
  \end{equation*}
The quantity is indeed bounded by
\begin{equation*}
  \begin{aligned}
    & \E_{\bX,Y} \left[ \left|\tJJ_* - \inf_{\phi_* \in \Sc} \tJJ(\phi_*) + \inf_{\phi_* \in \Sc} \tJJ(\phi_*) - \inf_{\phi_* \in \Sc} \hJJ(\phi_*) + \inf_{\phi_* \in \Sc} \hJJ(\phi_*) - \hJJ_*\right| \indic_{\bX,\bX' \in \Ac}\right] \\
    & \leq \E_{\bX,Y} \left[ \left(\inf_{\phi_* \in \Sc} \tJJ(\phi_*) - \tJJ_* \right) \indic_{\bX,\bX' \in \Ac}\right] + \E_{\bX,Y} \left[ \left(\inf_{\phi_* \in \Sc} \hJJ(\phi_*) - \hJJ_* \right) \indic_{\bX,\bX' \in \Ac}\right] \\
    & \quad + \E_{\bX,Y} \left[ \left|\inf_{\phi_* \in \Sc} \tJJ(\phi_*) - \inf_{\phi_* \in \Sc} \hJJ(\phi_*) \right| \indic_{\bX,\bX' \in \Ac}\right].
  \end{aligned}
  \end{equation*}
  The third term is finite since, by picking any $\phi_* \in \Sc$,
  \begin{equation*}
    \begin{aligned}
      \E_{\bX,Y} \left[ \left|\inf_{\phi_* \in \Sc} \tJJ(\phi_*) - \inf_{\phi_* \in \Sc} \hJJ(\phi_*) \right| \indic_{\bX,\bX' \in \Ac}\right] & \leq 2 \E_{\bX,Y} \left[ \left|\hJJ(\phi_*) \right| \indic_{\bX,\bX' \in \Ac}\right] \\
      & \leq \frac{2}{n} \sum_{i=1}^n \E_{\bX,Y}  \left[ \left|l(\phi_*,X_i) \right| \indic_{\bX,\bX' \in \Ac}\right] + 2 \RR(\phi_*),
    \end{aligned}
  \end{equation*}
  together with the assumption that $\E_{X} \left[ \left| l(\phi_*, X) \right|\right] < + \infty$ and the fact that $\RR(\phi_*) < + \infty$ since $\phi_*$ is a minimizer of $\JJ$.
  The first two terms can further be bounded in the same fashion so that only one is detailed. Due to the optimality of $\phi_*$ and leveraging Assumption \ref{ass:holder_l}
  \begin{equation*}
    \begin{aligned}
      \E_{\bX,Y} \left[ \left(\inf_{\phi_* \in \Sc} \hJJ(\phi_*) - \hJJ_* \right) \indic_{\bX,\bX' \in \Ac}\right] & \leq \E_{\bX,Y} \left[ \frac{1}{n} \sum_{i=1}^n \E_Z[a(Z,X_i)] \vartheta(\hSc ; \Sc)^\alpha  \indic_{\bX,\bX' \in \Ac}\right] \\
      & = \E_{\bX,Y} \left[ \E_Z[a(Z,X_i)] \vartheta(\hSc ; \Sc)^\alpha  \indic_{\bX,\bX' \in \Ac}\right] \\
      & \leq \|a\|_{L^2} \|\vartheta(\hSc ; \Sc)^\alpha \indic_{\bX \in \Ac} \|_{L^2},
    \end{aligned}
  \end{equation*}
  by an application of Cauchy-Schwartz inequality. 
  The latter quantity is finite since $\|a\|_{L^2} \leq 2 \|a\|_{\psi_1}$ and using the bounds on the moments of $\vartheta(\hSc ; \Sc)$ of Lemma \ref{lem:subexp_bound_distance}. 
  
  Therefore, in virtue of Fubini's theorem,
  \begin{equation*}
      \E_{\bX} \left[ \E_Y \left[\left(\tJJ_* - \hJJ_*\right) \indic_{Y \in I(\bX)} \right] \indic_{\bX \in \Ac}\right] = \E_{\bX,Y} \left[ \left(\tJJ_* - \hJJ_*\right) \indic_{\bX,\bX' \in \Ac}\right] = 0 \\
  \end{equation*}
  since $\bX$ and $\bX'$ share the same distribution.
\end{proof}

\begin{lemma} \label{lem:bound_E2}
  The term $E_2$ from the proof of Lemma \ref{lem:bound_conditional_expectation} can be bounded by
   \begin{equation*}
      E_2 \leq 2^{\max(0,\alpha-1)+2} \left( \|a\|_{\psi_1} \diam(\Sc)^\alpha + 2 \|a\|_{\psi_1}^{\frac{\beta}{\beta-\alpha}} \left(\tau^{-1} \max\left( \frac{4\alpha}{\beta - \alpha},1\right) \right)^\frac{\alpha}{\beta-\alpha} \right) \sqrt{p_n}.
    \end{equation*}
\end{lemma}
\begin{proof}
The term $e_2$ can further be bounded, for any $\phi_* \in \Sc$, by 
\begin{equation*}
  \begin{split}
  e_2 & = \E_Y\left[\left( l(\hphi,Y) - l(\hphi,X_1) \right) \indic_{Y \notin I(\bX)} \right] + \hJJ_* - \JJ_* \\
    & = e_{2,1} + e_{2,2},
  \end{split}
\end{equation*}
with 
\begin{equation*}
  \begin{aligned}
  e_{2,1} & = \E_Y\left[\left( l(\hphi,Y) - l(\hphi,X_1) - l(\phi_*,Y) + l(\phi_*,X_1) \right) \indic_{Y \notin I(\bX)} \right], \\
  e_{2,2} &= \E_Y\left[\left( l(\phi_*,Y) - l(\phi_*,X_1) \right) \indic_{Y \notin I(\bX)} \right] + \hJJ_* - \JJ_*.
    \end{aligned}
\end{equation*}
Therefore 
\begin{equation*}
 E_2 \leq E_{2,1} + E_{2,2} \quad \text{with } E_{2,i} = \E\left[\sup_{\hphi \in \hSc} e_{2,i} \, | \, \bX \in \Ac\right].
\end{equation*}
The remaining of the proof consists in bounding each term.

\begin{description}
  \item[{Bounding $E_{2,1}$}]
  The term $e_{2,1}$ can further be bounded by
 \begin{equation*}
      \begin{aligned}
        e_{2,1} &\leq \E_Y\left[ (l(\hphi,Y) - l(\phi_*,Y) + l(\phi_*,X_i) - l(\hphi,X_i)) \indic_{Y \notin I(\bX)} \right] & \\
        & \leq \E_Y\left[  a(Y,X_1) \indic_{Y \notin I(\bX)} \right] \vartheta(\hphi, \phi_*)^\alpha &  (\text{Assumption \ref{ass:holder_l}})  \\
        & \leq \E_Y\left[  a(Y,X_1) \indic_{Y \notin I(\bX)} \right] 2^{\max(0,\alpha-1)} \left( \vartheta(\hphi,  \Sc)^\alpha + \diam(\Sc)^\alpha\right). &
      \end{aligned}
  \end{equation*}
  This bound allows to upper bound $E_{2,1}$ by
    \begin{equation*}
      \begin{aligned}
        E_{2,1} &\leq 2^{\max(0,\alpha-1)} \E_{\bX} \left[  \E_Y\left[  a(Y,X_1) \indic_{Y \notin I(\bX)} \right]  \left( \vartheta(\hSc ;  \Sc)^\alpha + \diam(\Sc)^\alpha\right)  \,|\,  \bX \in \Ac \right]   \\
        & = 2^{\max(0,\alpha-1)} \frac{1}{1-p_n} \E_{\bX,Y}  \left[ \left(\vartheta(\hSc  ;\Sc)^\alpha + \diam(\Sc)^\alpha \right) a(Y,X_1) \indic_{\bX \in \Ac, \bX' \notin \Ac} \right],
      \end{aligned}
    \end{equation*}
  using Tonelli's theorem to swap the order of integration and \eqref{eq:prop_Ac_I}.
  The term 
  \begin{equation*}
        \begin{aligned}
          \E_{\bX,Y}  \left[  a(Y,X_1) \indic_{\bX \in \Ac, \bX' \notin \Ac} \right] & \leq \|a\|_{L^2} \P(\bX \in \Ac, \bX' \notin \Ac)^{1/2} \\
          & \leq 2 \|a\|_{\psi_1} \P(\bX \in \Ac, \bX' \notin \Ac)^{1/2},
        \end{aligned}
  \end{equation*}
  using Cauchy-Schwartz inequality.
  Now, two successive applications of the Cauchy-Schwartz inequality give
  \begin{equation*}
        \begin{aligned}
          \E_{\bX,Y} & \left[ \vartheta(\hSc  ;\Sc)^\alpha a(Y,X_1) \indic_{\bX \in \Ac, \bX' \notin \Ac} \right] & \\
          & \leq \| a \|_{L^4} \| \vartheta(\hSc  ;\Sc)^\alpha \indic_{\bX \in \Ac, \bX' \notin \Ac}\|_{L^4} \P(\bX \in \Ac, \bX' \notin \Ac)^{1/2} & \\
          & \leq\|a\|_{L^4} \| \vartheta(\hSc  ;\Sc)^{\beta-\alpha} \indic_{\bX \in \Ac, \bX' \notin \Ac} \|_{L^{\max\left(\frac{4\alpha}{\beta - \alpha},1\right)}}^{\frac{\alpha}{\beta-\alpha}} \P(\bX \in \Ac, \bX' \notin \Ac)^{1/2} & \\
          & \leq \|a\|_{L^4} \left(\tau^{-1}\| a \|_{L^{\max\left(\frac{4\alpha}{\beta - \alpha},1\right)}} \right)^{\frac{\alpha}{\beta-\alpha}} \P(\bX \in \Ac, \bX' \notin \Ac)^{1/2} &(\text{Lemma \ref{lem:subexp_bound_distance}})\\
          & \leq 4 \|a\|_{\psi_1} \left(\tau^{-1} \max\left( \frac{4\alpha}{\beta - \alpha},1\right) \| a \|_{\psi_1} \right)^{\frac{\alpha}{\beta-\alpha}} \P(\bX \in \Ac, \bX' \notin \Ac)^{1/2}. &
        \end{aligned}
    \end{equation*}
    On the other hand, the Lemma \ref{lem:negative_correlation} allows to show that
    \begin{equation*}
        \frac{\sqrt{\P(\bX \in \Ac, \bX' \notin \Ac)}}{1-p_n} \leq \frac{\sqrt{p_n (1-p_n)}}{1-p_n} \leq 2 \sqrt{p_n},
    \end{equation*}
since  $\sqrt{p_n/(1-p_n)} \leq 2\sqrt{p_n}$ for $p_n \leq 3/4$.
    Combining all these bounds leads to
    \begin{equation*}
      \begin{aligned}
      E_{2,1} &\leq 2^{\max(0,\alpha-1)+1} \sqrt{p_n} \left( \|a\|_{\psi_1} \diam(\Sc)^\alpha + 2 \|a\|_{\psi_1}^{\frac{\beta}{\beta-\alpha}} \left(\frac{1}{\tau} \max\left( \frac{4\alpha}{\beta - \alpha},1\right) \right)^\frac{\alpha}{\beta-\alpha} \right).
      \end{aligned}
    \end{equation*}

  \item[{Bounding $E_{2,2}$}]
  Since $\hJJ_* \leq \hJJ(\phi_*)$ for all $\phi_* \in \Sc$ the term $e_{2,2}$ can further be bounded by
\begin{equation*}
  \begin{aligned}
  e_{2,2} &\leq \E_Y\left[\left( l(\phi_*,Y) - l(\phi_*,X_1) \right) \indic_{Y \notin I(\bX)} \right] + \hJJ(\phi_*) - \JJ_*.
  \end{aligned}
\end{equation*}
Since $\Ac$ is invariant under permutations of the components of $\bX$, 
\begin{equation*}
\E\left[\hJJ(\phi_*) \, | \, \bX \in \Ac \right] = \frac{1}{n} \sum_{i=1}^n \E\left[\hJJ(\phi_*) \, | \, \bX \in \Ac \right] = \E\left[l(\phi_*,X_1) \, | \, \bX \in \Ac \right].
\end{equation*}
Therefore,
\begin{equation*}
  \begin{aligned}
    E_{2,2} & \leq \frac{1}{1-p_n} \E_{\bX} \left[  \E_Y\left[\left(l(\phi_*,X_1) - l(\phi_*,Y) \right) \indic_{Y \in I(\bX)} \right] \indic_{\bX \in \Ac} \right].
  \end{aligned}
\end{equation*}

Moreover, since $\E_Y\left[|l(\phi_*,Y)|\right] < +\infty$, the order of integration can be switched in virtue of Fubini's theorem, leading to
\begin{equation*}
    E_{2,2} \leq \frac{1}{1-p_n} \E_{\bX,Y} \left[ \left( l(\phi_*,X_1) - l(\phi_*,Y) \right) \indic_{\bX \in \Ac,\bX' \in \Ac} \right] = 0,
\end{equation*}
since $\bX$ and $\bX'$ have the same distribution.
\end{description}
Combining the bounds on $E_{2,1}$ and $E_{2,2}$ gives the desired result.
\end{proof}

\begin{lemma} \label{lem:subexp_bound_distance}
    Let $\Ac $ be the subset defined by \eqref{eq:def_Ac} and assume Assumptions \labelcref{ass:existstence_minimizer,ass:local_holder_error,ass:convergence_bassin,ass:lojasewicz_modulus,ass:holder_l}  hold, then conditionally on $\Ac$
    \begin{equation*}
      \vartheta(\hSc;\Sc)^{\beta - \alpha} \leq \frac{1}{\tau n} \sum_{i=1}^n \E_Y[a(Y,X_i)].
    \end{equation*}
    Furthermore, for all $p \geq 1$
    \begin{equation*}
      \left\|\vartheta(\hSc;\Sc)^{\beta - \alpha} \indic_{\bX \in \Ac} \right\|_{L^p} \leq \frac{\| a\|_{L^p}}{\tau}.
    \end{equation*}
\end{lemma}

\begin{proof}
The conditioning on $\bX \in \Ac$ implies that all $\hphi \in \hSc$ belong to the level-set $[\JJ \leq \JJ_0]$.
Let $\hphi \in \hSc$, the Assumption \ref{ass:local_holder_error} gives 
    \begin{equation*}
      \begin{aligned}
      \tau \vartheta(\widehat{\phi},\Sc)^\beta & \leq \JJ(\hphi) - \JJ_* \\
       & \leq \inf_{\phi_* \in \Sc} \JJ(\hphi) - \JJ_* + \hJJ(\phi_*) - \hJJ(\hphi) & (\text{since } \hphi \in \hSc) \\
       & = \inf_{\phi_* \in \Sc} \frac{1}{n} \sum_{i=1}^n \E_Y\left[ l(\hphi,Y) - l(\phi_*,Y)+ l(\phi_*,X_i) - l(\hphi,X_i)\right] & (\text{by Assumption \ref{ass:holder_l}}) \\
       & \leq \frac{1}{n} \sum_{i=1}^n \E_Y\left[ a(Y,X_i) \right] \inf_{\phi_* \in \Sc} \vartheta(\hphi,\phi_*)^\alpha.
      \end{aligned}
    \end{equation*}
    Therefore,
    \begin{equation*}
      \tau \vartheta(\widehat{\phi},\Sc)^{\beta-\alpha} \leq \frac{1}{n} \sum_{i=1}^n \E_Y\left[ a(Y,X_i) \right],
    \end{equation*}
    holds for any $\hphi \in \hSc$ leading to the expected result.
    The bound on the moments is obtained remarking that
    \begin{equation*}
      \begin{aligned}
        \left\| \frac{1}{\tau n} \sum_{i=1}^n \E_Y[a(Y,X_i)]  \indic_{\bX \in \Ac} \right\|_{L^p} & \leq \left\| \frac{1}{\tau n} \sum_{i=1}^n \E_Y[a(Y,X_i)] \right\|_{L^p} \\
        & \leq \frac{1}{\tau n}  \sum_{i=1}^n \left\| \E_Y[a(Y,X_i)]\right\|_{L^p} \\
        & \leq \frac{1}{\tau n}  \sum_{i=1}^n \left( \E_{\bX,Y}[a(Y,X_i)^p]\right)^{1/p},
       \end{aligned}
    \end{equation*}
    using Jensen's inequality. The final result is obtained remarking that the samples $(X_i)_{i=1}^n$ are independent copies.
\end{proof}

\subsection{Negative Correlation of \texorpdfstring{$\bX \in \Ac$ and $\bX \in \Ac'$}{events}}

\begin{lemma} \label{lem:negative_correlation}
Let $\bX = (X_1, \ldots, X_n)$ and $\bX' = (Y,X_2, \ldots, X_n)$ where $(X_i)_{i=1}^n$ and $Y$ are independent samples of $\mu$. Let $\Ac $ be the subset of $\Xsp^n$ defined in \eqref{eq:def_Ac} then
\begin{equation*}
  \P( \bX \in \Ac, \bX' \notin \Ac) \leq \P(\bX \in  \Ac) \P(\bX \notin  \Ac).
\end{equation*}
\end{lemma}
\begin{proof}
  The probability of the intersection can be decomposed as follows
  \begin{equation*}
    \begin{aligned}
      \P( \bX \in \Ac, \bX' \notin \Ac) = \E\left[ \indic_{\bX \in \Ac} \indic_{\bX' \notin \Ac}  \right] &= \E\left[ \indic_{\bX \in \Ac}\right] \E\left[ \indic_{\bX' \notin \Ac}  \right] + \Cov\left(\indic_{\bX \in \Ac},\indic_{\bX' \notin \Ac}\right).
    \end{aligned}
  \end{equation*}
  Let $\bZ = (X_2,\ldots,X_n)$. In virtue of the law of total covariance 
  \begin{equation*} 
  \begin{aligned}
    \Cov\left(\indic_{\bX \in \Ac},\indic_{\bX' \notin \Ac} \right) &= \E\left[ \Cov\left( \indic_{\bX \in \Ac},\indic_{\bX' \notin \Ac} | \bZ \right) \right] + \Cov\left( \E\left[\indic_{\bX \in \Ac} | \bZ \right] , \E\left[\indic_{\bX' \notin \Ac} | \bZ \right] \right).
  \end{aligned}
  \end{equation*}
  Since conditionally on $\bZ$, the random variables $X_1$ and $Y$ are independent, the covariance $\Cov\left( \indic_{\bX \in \Ac},\indic_{\bX' \notin \Ac} | \bZ \right) = 0$. 
  Let $g : \Xsp^{n-1} \to \R$ be the function defined by $g(\bZ) = \E\left[\indic_{\bX' \notin \Ac} | \bZ \right]$. Then
  \begin{equation*}
    \Cov\left( \E\left[\indic_{\bX \in \Ac} | \bZ\right] , \E\left[\indic_{\bX' \notin \Ac} | \bZ \right] \right) = \Cov\left( g(\bZ), 1-g(\bZ) \right) \leq 0,
  \end{equation*}
  by the FKG inequality \citep{fortuin1971correlation}. Therefore $\Cov\left(\indic_{\bX \in \Ac},\indic_{\bX' \notin \Ac} \right) \leq 0$, which gives the desired result.
\end{proof}

\section{Additional Proofs for Applications}

\subsection{Concentration of Sample Covariance Matrices} 

\begin{lemma} \label{lem:sample_covariance_subexp}
  Let $Z \in \R^d$ be a sub-Gaussian random vector and let $(Z_i)_{i=1}^n$ be $n$ independent copies of $Z$.
  Let $A = \E[ZZ^T]$ and $\widehat{A} = \frac{1}{n} \sum_{i=1}^n Z_i Z_i^T$ then
  \begin{equation*}
    \P\left( \|A - \widehat{A}\| > t \right) \leq 2 d \exp\left( -\frac{nt^2}{2
    \left(e^2 (d+1)^{2} \|Z\|_{\psi_2}^4 + e (d+1) \|Z\|_{\psi_2}^2 t\right) }\right).
  \end{equation*}
\end{lemma}

\begin{proof}
  This result is a consequence of matrix Bernstein inequalities \citep[Theorem 6.2]{tropp2012user} stated in the next theorem.
  \begin{theorem} \label{thm:matrix_bernstein}
    Let $(S_i)_{k=1}^n$ be $n$ independent, self-adjoint symmetric random matrices of dimension $d$. Assume that 
    $\E[S_i] = 0$ and that there exist $R$ and $n$ matrices $(B_i)_{i=1}^n$ such that $\E[S_i^p] \preceq \frac{p!}{2} R^{p-2} B_i^2$ for all $p \geq 2$. Then
    \begin{equation*}
      \P\left( \lambda_{\max} \left( \sum_{i=1}^n S_i\right) > t \right) \leq d \exp\left( -\frac{t^2}{2(\sigma^2 + Rt) }\right),
    \end{equation*}
    with $\sigma^2 = \| \sum_{i=1}^n B_i^2\|$ and where $\lambda_{\max}$ stands for the algebraically largest eigenvalue. 
  \end{theorem}
  This result will be used by setting $S_i = \frac{1}{n} \left( Z_i Z_i^T - \E[ZZ^T] \right)$. This choice implies that $\E[S_i] = 0$ and $\sum_{i=1}^n S_i = \widehat{A} - A$. Moreover, since $S_i$ are self-adjoint matrices, the spectral norm of the matrix $\widehat{A} - A$ is equal to its largest magnitude eigenvalue that is $\| \widehat{A} - A \| = \max(-\lambda_{\min}(\widehat{A} - A), \lambda_{\max}(\widehat{A} - A))$ (here $\lambda_{\min}$ stands for the algebraically smallest eigenvalue).
  Since each $S_i \preceq \|S_i\| \Id_d$, the Bernstein condition can be obtained by bounding the moments of $\|S_i\|$ because $\E[S_i^p] \leq \E\left[\|S_i\|^p\right] \Id_d$. This is achieved remarking that
  \begin{equation*}
    \begin{aligned}
    \|S_i\| &= \sup_{v \in \mathbb{S}^{d-1}} |\langle S_iv,v \rangle| = \frac{1}{n} \sup_{v \in \mathbb{S}^{d-1}} \left| \langle Z_i, v\rangle^2 - \E\left[\langle Z, v\rangle^2 \right]\right| \\
    & \leq \frac{1}{n} \max\left( \|Z_i\|_2^2, \sup_{v \in \mathbb{S}^{d-1}} \E[\langle Z, v\rangle^2] \right) \\
    & \leq \frac{1}{n} \max\left( \|Z_i\|_2^2, \| Z \|_{\psi_2}^2 \right),
    \end{aligned}
  \end{equation*}
since $\E[\langle Z, v\rangle^2] \leq \| \langle Z, v\rangle^2 \|_{\psi_1} =  \| \langle Z, v\rangle \|_{\psi_2}^2 $ and using the definition of the sub-Gaussian norm of vectors.
  This implies that 
  \begin{equation*}
    \begin{aligned}
    \| \|S_i \| \|_{\psi_1} &\leq \frac{1}{n} \left( \| \|Z_i\|_2^2 \|_{\psi_1} + \| Z \|_{\psi_2}^2 \right) \\
    & = \frac{1}{n} \left( \| \|Z_i\|_2 \|_{\psi_2}^2 + \| Z \|_{\psi_2}^2 \right) \\
    & \leq \frac{d+1}{n}  \| Z \|_{\psi_2}^2,  \\
    \end{aligned}
  \end{equation*}
leading to $\E[\| S_i\|^p] \leq \frac{p!}{2} \left(e \frac{d+1}{n} \|Z\|_{\psi_2}^2\right)^p$.
  The Theorem \ref{thm:matrix_bernstein} can now be applied to $\lambda_{\max}(\widehat{A} - A)$ and $-\lambda_{\min}(\widehat{A} - A) = \lambda_{\max}(A - \widehat{A})$ with $R = e \frac{d+1}{n} \|Z_i\|_{\psi_2}^2$ and $B_i = e \frac{d+1}{n} \|Z_i\|_{\psi_2}^2 \Id_d$ with the corresponding $\sigma^2 = \frac{1}{n} e^2 (d+1)^{2} \|Z_i\|_{\psi_2}^4$. Applying the union bound to these estimates completes the proof.
\end{proof}

\subsection{Largest Eigenvector of Covariance Matrices} \label{sec:proof_eigenvector}

\begin{proof}[Proof of Lemma \ref{lem:assumption_least_eigen}]
Prior to the proof, general facts on $\mathbb{S}^{d-1}$ as a Riemannian manifold are introduced.

The tangent space $\TT_\phi \mathbb{S}^{d-1} = \{ v \in \R^d \, | \, \langle \phi, v \rangle_{\R^d} = 0 \}$ and the metric is defined by $\g_\phi(u,v) = \langle u,v\rangle_{\R^d}$, the Euclidean scalar product, for all $u,v \in \TT_\phi \mathbb{S}^{d-1}$. The norm of vectors in $\TT_\phi \mathbb{S}^{d-1}$ will be denoted by $\|\cdot\|_\phi = \g_\phi(\cdot,\cdot)$. In particular, here $\|\cdot\|_\phi = \|\cdot\|_2$ for all $\phi \in \mathbb{S}^{d-1}$ so that the dependency on $\phi$ might be dropped.
With this metric, the logarithm map reads $\log_\phi(\psi) = \left(\psi - \langle \psi,\phi\rangle \phi \right) \frac{\theta}{\sin(\theta)}$ with $\theta = \arccos\langle \phi,\psi\rangle$. 
The exponential map is given by $\exp_\phi(\xi) = \cos(\|\xi\|_2) \phi + \sin(\|\xi\|) \frac{\xi}{\|\xi\|_2}$ for $\xi \in \TT_\phi \mathbb{S}^{d-1}$.
Furthermore, the geodesic from $\phi$ to $\psi$ reads $\gamma(t) = \cos(\|\log_\phi(\psi) t\|) \phi + \sin(\|\log_\phi(\psi)t\|)\frac{\log_\phi(\psi)}{\|\log_\phi(\psi)\|}$. The geodesic distance between $\phi$ and $\psi$ is therefore $\vartheta(\phi,\psi) = \arccos \langle \phi, \psi \rangle_{\R^d}$.
One can check that $\| \phi - \psi \|_2 \leq \vartheta (\phi,\psi) \leq \frac{\pi}{2} \| \phi - \psi \|_2$.
The proof consists in checking the assumptions one by one. To simplify notation the matrix $A$ will refer to $\Cov(\mu)$.

\begin{description}
  \item[Assumption \ref{ass:existstence_minimizer}] The functions $\JJ$ and $\hJJ$ are continuous and defined on a compact domain $\Msp = \mathbb{S}^{d-1}$. Therefore, each of them admits at least one minimizer. 
  \item[Assumption \ref{ass:local_holder_error}]
  Let $\phi \in \Sc = \{ \pm u_1 \}$ where $u_1$ refer to the eigenvector associated to the largest eigenvalue $\lambda_1$.
  Remark that any point on the sphere is at distance at most $\pi/2$ of one of the minimizers, so that \ $\vartheta(\psi,\Sc) \leq \pi/2$ for all $\psi \in \mathbb{S}^{d-1}$.

  For a given $\psi \in \mathbb{S}^{d-1}$ we fix $\phi$ as the closest element of $\Sc$ to $\psi$ (since $\Sc = \{ \pm u_1\})$.
  Let $\xi = \log_\phi(\psi)$ and one has that $\|\xi\|_2 = \| \xi \|_{\phi} = \vartheta(\psi,\phi) \leq \frac{\pi}{2}$.
  The bound is obtained using
  \begin{equation*}
    \begin{split}
      & \JJ(\exp_{\phi}(\xi)) - \JJ(\phi) \\
      & = -\cos^2(\|\xi\|) \langle A \phi, \phi \rangle - 2 \cos(\|\xi\|) \sin(\|\xi\|) \left \langle A \phi, \frac{\xi}{\|\xi\|} \right \rangle - \sin^2(\|\xi\|) \left \langle A \frac{\xi}{\|\xi\|},\frac{\xi}{\|\xi\|} \right\rangle.
    \end{split}
  \end{equation*}
  Observing that $\langle A \phi, \xi\rangle = \langle \Lambda U^T \phi, U^T \xi\rangle = \lambda_d \langle \pm u_1,\xi \rangle = 0$ since $\xi \in \TT_\phi \mathbb{S}^{d-1}$ we obtain
  \begin{equation*} 
    \begin{split}
      \JJ(\exp_{\phi}(\xi)) - \JJ(\phi) & = \frac{\sin^2(\|\xi\|_2)}{\|\xi\|_2^2} \left( \| \xi \|_2^2 \langle A \phi, \phi \rangle - \langle A \xi,\xi \rangle  \right) \\
      & \geq \frac{4}{\pi^2} (\lambda_{1} - \lambda_2) \|\xi\|_2^2.
    \end{split}
  \end{equation*}
  The last inequality is obtained using Lemma \ref{lem:lower_bound_eigenvalue} and from the fact that the function $t \mapsto \sin^2(t)/t^2$ is lower bounded by $\frac{4}{\pi^2}$ on $[0,\frac{\pi}{2}]$. Since the norm of the tangent vector $\|\xi\|_2$ is equal to the geodesic distance $\vartheta(\psi,\Sc)$, this lead to
  \begin{equation*}
    \frac{4}{\pi^2} (\lambda_{1} - \lambda_2) \vartheta(\psi,\Sc)^2 \leq \JJ(\psi) - \JJ_*,
  \end{equation*}
  for all $\psi \in \Msp$. So that Assumption \ref{ass:local_holder_error} holds for $\JJ_0 = +\infty$, $\beta = 2$ and $\tau = \frac{4}{\pi^2} (\lambda_{1} - \lambda_2)$.

  \item[Assumption \ref{ass:convergence_bassin}]
  Since Assumption \ref{ass:local_holder_error} holds for $\JJ_0 = + \infty$, Assumption \ref{ass:convergence_bassin} holds with $\eta = 0$ since $\P(\JJ(\hphi) > \JJ_0) = 0$.

  \item[Assumption \ref{ass:lojasewicz_modulus}] 
  Let $\widehat{\lambda}_1 \geq \ldots \geq \widehat{\lambda}_m$ be the eigenvalues of $\widehat{A} = \frac{1}{n} \sum_{i=1}^n X_i X_i^T$. 
  The derivation of Assumption \ref{ass:local_holder_error} can also be carried on $\hJJ$ so that 
  \begin{equation*}
    \frac{4}{\pi^2} (\widehat{\lambda}_{1} - \widehat{\lambda}_2) \vartheta(\psi,\hSc)^2 \leq \hJJ(\psi) - \hJJ_*,
  \end{equation*}
  holds for $\phi \in \Msp$. This bound can be used as long as $(\widehat{\lambda}_{1} - \widehat{\lambda}_2) > 0$.

  In virtue of the Weyl's inequality the distance between the eigenvalues can be bounded by $|\lambda_i - \widehat{\lambda}_i| \leq \| A - \widehat{A}\|$ for all $1 \leq i \leq d$ so that the eigengap $\widehat{\lambda}_{1} - \widehat{\lambda}_2$ can be lower bounded by
  \begin{equation*}
    \begin{aligned}
    \widehat{\lambda}_{1} - \widehat{\lambda}_2 & \geq \lambda_1 - \lambda_2 - |\widehat{\lambda}_{1} - \lambda_1| - |\widehat{\lambda}_{2} - \lambda_2| \\
    & \geq \lambda_1 - \lambda_2 - 2 \| A - \widehat{A}\|.
    \end{aligned}
  \end{equation*}
  When $2 \| A - \widehat{A}\| \leq \frac{1}{2} (\lambda_1 - \lambda_2)$ the eigengap $\widehat{\lambda}_{1} - \widehat{\lambda}_2 \geq \frac{1}{2} (\lambda_1 - \lambda_2)$.
  This event happens with high probability since by Lemma \ref{lem:sample_covariance_subexp}:
  \begin{equation*}
    \P\left( \| A - \widehat{A}\| > \frac{1}{4} (\lambda_1 - \lambda_2)  \right) \leq 2 d e^{-\frac{n(\lambda_1 - \lambda_2)^2}{32e^2 (d+1)^{2} \|X\|_{\psi_2}^4 + 8e(d+1)\|X\|_{\psi_2}^2(\lambda_1 - \lambda_2)}}.
  \end{equation*}

  \item[Assumption \ref{ass:holder_l}] 

  Let $\phi,\psi \in \Msp$ and $x,y \in \Xsp$ such that $x \neq y$, we have
  \begin{equation*}
    \begin{aligned}
      & l(\phi,x) - l(\psi,x) - l(\phi,y) + l(\psi,y) \\
      &= \langle x, \phi - \psi\rangle \langle x, \psi + \phi \rangle + \langle y, \psi-\phi\rangle \langle y, \psi + \phi\rangle  \\
      & \leq 2\left( \|x\|_2^2 + \|y\|_2^2 \right) \|\phi-\psi\|_2 \\
      & \leq 2\left( \|x\|_2^2 + \|y\|_2^2\right) \vartheta(\phi,\psi),
    \end{aligned}
  \end{equation*}
  using $(a^2-b^2) = (a-b)(a+b)$ for the first equality and where the first inequality is obtained with the Cauchy-Schwartz inequality together with the fact that $\psi$ and $\phi$ have unit norm.
  The function defined by
\begin{equation*}
    a(x,y) = 2\left( \|x\|_2^2 + \|y\|_2^2\right) \indic_{x \neq y},
\end{equation*}
for $x,y \in \Xsp$, is a pseudometric and therefore is a candidate for Assumption \ref{ass:holder_l} to hold.
Since $X$ and $Y$ share the same distribution, the sub-exponential norm of $a$ can be bounded by
  \begin{equation*}
    \begin{aligned}
      \| a \|_{\psi_1} &\leq 4 \|\|X\|^2\|_{\psi_1} = 4 \||X\|\|_{\psi_2}^2.
    \end{aligned}
  \end{equation*}
  Finally, using the fact that the sub-Gaussian norm $\| \| X \|\|_{\psi_2} \leq \sqrt{d} \|X\|_{\psi_2}$, the Assumption \ref{ass:holder_l} is therefore satisfied with $\Upsilon = \Msp$, $\alpha = 1$ and the function $a$ that is sub-exponential when $X$ is sub-Gaussian.
\end{description} 
  
\end{proof}

\begin{lemma} \label{lem:lower_bound_eigenvalue}
Let $A \in \R^{d \times d}$ be a real symmetric matrix. Let $U = (u_1,\ldots,u_d) \in \R^{d \times d}$ and $\lambda_1 \geq \ldots \geq \lambda_d $ such that $A = U \Lambda U^T$ with $\Lambda = \mathrm{diag} (\lambda_i)_{i=1}^d$.
Let $v \in \R^d$ orthogonal to $u_1$ then
\begin{equation*}
  \|v\|_2^2 \langle A u_1,u_1 \rangle - \langle A v, v\rangle  \geq (\lambda_{1} - \lambda_2) \|v\|_2^2.
\end{equation*}  
\end{lemma}

\begin{proof}
Since $v$ is orthogonal to $u_1$,
\begin{equation*}
\begin{split}
\|v\|_2^2 \langle A u_1,u_1 \rangle - \langle A v, v\rangle  &= \|v\|_2^2 \langle \Lambda U^T u_1, U^Tu_1 \rangle - \langle \Lambda U^T v, U^T v\rangle\\
& = \|v\|_2^2 \lambda_1 - \sum_{i=2}^{d} \lambda_i \langle u_i, v\rangle^2.
\end{split}
\end{equation*}
The vector $v$ can be decomposed in the orthogonal basis as $v = \sum_{i=2}^{d} \langle u_i,v\rangle u_i$ and therefore has a norm $\|v\|_2^2 = \sum_{i=2}^{d} \langle u_i, v\rangle^2$. We therefore conclude by remarking that
\begin{equation*}
\begin{split}
\|v\|_2^2 \lambda_1 - \sum_{i=2}^{d} \lambda_i \langle u_i, v\rangle^2 \geq \lambda_{1} \|v\|_2^2 - \lambda_2 \|v\|_2^2.
\end{split}
\end{equation*}
\end{proof}

\subsection{LASSO} \label{sec:proof_lasso}

\begin{proof}[Proof of Lemma \ref{lem:assumption_lasso}]

The proof consists in checking the assumptions one by one.

\begin{description}
  \item[Assumption \ref{ass:existstence_minimizer}] Since the functions $\JJ$ and $\hJJ$ are convex and coercive, each of them admits at least one minimizer.
  \item[Assumption \ref{ass:local_holder_error}]

In order to apply Lemma \ref{lem:hoffman}, the function $\JJ$ has to be written in the form $\JJ(\phi) = \LL(\phi) + c$ where $\LL$ is of the form $\LL(\phi) = \frac{1}{2} \|A \phi - b\|_2^2 + \lambda \| \phi\|_1$ and $c$ is a constant term.
Reminding that  $A = \E\left[ \theta(Y) \theta(Y)^T \right]^{\frac{1}{2}}$, it can be shown that 
\begin{equation*}
  \begin{split}
  \ff(\phi) &= \frac{1}{2} \left \langle A \phi, A \phi \right\rangle - \langle \phi, \E_{\mu}[\theta(Y) V] \rangle +  \frac{1}{2} \E[V^2] \\ 
  & =\frac{1}{2} \left \langle A (\phi-\phi_0), A(\phi-\phi_0) \right\rangle + \frac{1}{2} \E[V^2] - \frac{1}{2}\| A \phi_0 \|_2^2,
  \end{split}
  \end{equation*}
with $\phi_0$ being chosen as $A A \phi_0 =  \E[\theta(Y)V]$. Such a vector always exists since it is a solution of the linear least square problem
\begin{equation}\label{eq:least_square_lasso}
  \min_{\phi \in \R^m} \ff(\phi).
\end{equation}
Minimizing $\JJ$ is therefore equivalent to minimizing $\LL$ defined by
\begin{equation*} 
  \LL(\phi) = \frac{1}{2} \left \| A \phi - b \right\|_2^2 + \lambda \| \phi \|_1,
\end{equation*}
setting $b = A \phi_0$. Remark at this point that since $\phi_0$ is the solution of \eqref{eq:least_square_lasso} 
\begin{equation*}
  0 \leq \ff(\phi_0) = \frac{1}{2} \E[V^2] - \frac{1}{2}\| A \phi_0 \|_2^2,
\end{equation*}
so that $\|b\|_2^2 \leq \E[V^2]$.

Finally, observe that $\lambda \| \phi_*\|_1 \leq \JJ(\phi_*) \leq \JJ(0) = \ff(0) = \frac{1}{2} \E[V^2]$. Thus, setting $\JJ_0 = \E[V^2]$ and $R = \E[V^2]/\lambda$ ensures that $\JJ(\phi) \leq \JJ_0$ implies that $\|\phi\|_1 \leq R$.

The Lemma \ref{lem:hoffman} can now be applied with these matrix $A$ and value $R$ leading to a bound
\begin{equation*}
  \tau \vartheta(\phi,\Sc)^2 \leq  \LL(\phi) - \LL_* = \JJ(\phi) - \JJ_*,
\end{equation*}
as long as $\| \phi \|_1 \leq R$ for $\tau$ with expression given in \eqref{eq:tau_lasso}.
Using furthermore the fact that $\|b\|_2^2 \leq \E[V^2]$, 
\begin{equation}\label{eq:lower_bound_tau_lasso}
\begin{split}
  \tau & \geq \left(4 H^2 \left[1 + \lambda R + \left(R \|A\| + \sqrt{\E[V^2]}\right)\left(4 R \|A\| + \sqrt{\E[V^2]}\right)\right]\right)^{-1} \\
  & = \left(4 H^2\left[1 + \lambda R + (R \|A\| + \sqrt{\lambda R})(4 R \|A\| + \sqrt{\lambda R})\right]\right)^{-1}.
\end{split}
\end{equation}
This lower-bound will be used in place of $\tau$. This choice is not optimal, since it was used that $\|b\|_2^2 \leq \E[V^2]$, but it will ease the proof of Assumption \ref{ass:lojasewicz_modulus}.

  \item[Assumption \ref{ass:convergence_bassin}]
  Similarly to the proof of Assumption \ref{ass:local_holder_error},  $\hJJ(\hphi) \leq \hJJ(0) = \frac{1}{2} \frac{1}{n} \sum_{i=1}^n V_i^2$ so that $\JJ(\hphi) \leq \JJ_0$ as soon as $\frac{1}{n} \sum_{i=1}^n V_i^2 \leq 2 \E[V^2]$.
  This event happens with high probability since using the Bernstein inequality of Theorem \ref{thm:bernstein_mcdiarmid},
  \begin{equation} \label{eq:concentration_Vi}
  \P\left( \left| \frac{1}{n} \sum_{i=1}^n V_i^2 - \E[V^2] \right| > t\right) \leq 2\exp\left( - \frac{nt^2}{2(e^2 \| V^2\|_{\psi_1}^2 + e \| V^2\|_{\psi_1}t)}\right),
  \end{equation}
  for all $t > 0$, because the random vector $V^2$ is sub-exponential with $\|V^2\|_{\psi_1} = \|V\|_{\psi_2}^2$. 
  Choosing $t = \E[V^2]$ therefore shows that the assumption is fulfilled with
  $\eta(n) = 2e^{-c_1 n}$ with 
\begin{equation*}
c_1 = \frac{\|V\|_2^4}{2(e^2 \| V\|_{\psi_2}^4 + e \| V\|_{\psi_2}^2\|V\|_2^2)}.
\end{equation*}

  \item[Assumption \ref{ass:lojasewicz_modulus}]
  The assumption will be proved checking that
  \begin{equation*}
    \widetilde{\tau} = \inf \left\{ \frac{\hJJ(\phi) - \hJJ_*}{\vartheta(\phi,\hSc)^2}, \quad \phi \in [\JJ \leq \JJ_0]\right\}
  \end{equation*} 
  is greater that $\frac{1}{2} \tau$ with high probability. 
  This will imply that Assumption \ref{ass:lojasewicz_modulus} holds since $\htau \geq \widetilde{\tau}$.

  In this proof, the matrix $A$ will also refer to $\E\left[ \theta(Y)\theta(Y)^T \right]^\frac{1}{2}$ and $\widehat{A}$ to its empirical counterpart.

  Note that $\widetilde{\tau}$ is defined in a similar fashion as \eqref{eq:lower_bound_tau_lasso} by
    \begin{equation} \label{eq:lower_bound_tau_lasso_empirical}
    \widetilde{\tau} = \left(4 \widehat{H}^2\left[1 + \JJ_0 + \left(\JJ_0 \|\widehat{A}\| \lambda^{-1} + \sqrt{\JJ_0}\right)\left(4 \JJ_0 \|\widehat{A}\| \lambda^{-1} + \sqrt{\JJ_0}\right)\right]\right)^{-1},
  \end{equation}
  so that the high-probability bound will be obtained by proving that $\widehat{H}^2 \leq 2^{1/2} H^2$ and $\| \widehat{A} \| \leq 2^{1/4} \| A \|$ with high probability.
  If all these inequalities hold, then  
  \begin{equation*}
    \begin{split}
    \widetilde{\tau} &\geq  \left(4 \sqrt{2} H^2\left[1 + \JJ_0 + \left(2^{1/4} \JJ_0 \|A\| \lambda^{-1} + \sqrt{\JJ_0}\right)\left(2^{1/4} 4 \sqrt{\JJ_0} \|A\| \lambda^{-1} + \sqrt{\JJ_0}\right)\right]\right)^{-1} \\
    &\geq  \left(4 \sqrt{2} H^2\left[\sqrt{2}(1 + \JJ_0) + \sqrt{2}\left(\JJ_0 \|A\| \lambda^{-1} + \sqrt{\JJ_0}\right)\left(4 \JJ_0 \|A\| \lambda^{-1} + \sqrt{\JJ_0}\right)\right]\right)^{-1} \\
    &\geq  \frac{1}{2} \tau.
    \end{split}
  \end{equation*}

  \begin{itemize}
    \item Proving that $\widehat{H}^2 \leq 2^{1/2} H^2$. Using definition of $H$, the Hoffman's constant of the matrix $C(A,\lambda)$ as defined in \eqref{eq:hauffman_matrix}, and under the assumption that it is positive, it can be proved that 
    \begin{equation} \label{eq:lower_bound_submatrices}
      \lambda_{|I|}(C(A,\lambda)_I^TC(A,\lambda)_I) = \sigma_{|I|}^2(C(A,\lambda)_I) \geq \frac{1}{H^2},
    \end{equation}
    for all subsets of rows $I \subseteq \{ 1,\ldots 2^m + m + 2\}$.
    In virtue of Weyl's inequality, for all $1 \leq i \leq |I|$ and all subsets $I$,
    \begin{equation*}
    \begin{split}
      r_{I,i} &= |\lambda_{i}(C(A,\lambda)_I^TC(A,\lambda)_I) - \lambda_{i}(C(\widehat{A},\lambda)_I^TC(\widehat{A},\lambda)_I)| \\
      & \leq \| C(A,\lambda)_I^TC(A,\lambda)_I - C(\widehat{A},\lambda)_I^TC(\widehat{A},\lambda)_I \| \\
      & \leq \| C(A,\lambda)^TC(A,\lambda) - C(\widehat{A},\lambda)^TC(\widehat{A},\lambda) \|,\\
    \end{split}
    \end{equation*}
by the interlacing property of eigenvalues. Further computations show that
\begin{equation*}
  \begin{aligned}
    r_{I,i} & \leq \left\| \begin{pmatrix} A A - \widehat{A} \widehat{A} & 0 \\ 0 & 0 \end{pmatrix}\right\|
    = \left\| \E[\theta(Y) \theta(Y)^T] - \frac{1}{n}\sum_{i=1}^n \theta(y_i)\theta(y_i)^T \right\|.
  \end{aligned}
\end{equation*}

In particular, under the event $\| A A - \widehat{A} \widehat{A} \| \leq \frac{1}{\epsilon H^2}$,
  \begin{equation} \label{eq:final_bound_submatrices}
    \begin{aligned}
    \lambda_{|I|}(C(\widehat{A},\lambda)_I^TC(\widehat{A},\lambda)_I) & \geq \lambda_{|I|} (C(A,\lambda)_I^TC(A,\lambda)_I) - \frac{1}{\epsilon H} \\
    & \geq \left(1 - \frac{1}{\epsilon}\right) \lambda_{|I|} (C(A,\lambda)_I^TC(A,\lambda)_I),
    \end{aligned}
  \end{equation}
  using \eqref{eq:lower_bound_submatrices}. This implies that under this event, the submatrices $C(\widehat{A},\lambda)_I$ share the row rank of their counterpart $C(A,\lambda)_I$. Therefore, the set of constraints in the maximization problem \eqref{eq:def_hoffman} are equal. The equation \eqref{eq:final_bound_submatrices} furthermore implies that $\widehat{H}^2 \leq \frac{\epsilon}{\epsilon-1} H^2$. Choosing $\epsilon = \frac{2^{1/2}}{2^{1/2}-1}$ gives the desired bound.

    \item Proving that $\| \widehat{A} \| \leq 2^{1/4} \| A \|$. This is equivalent to proving that $\| \widehat{A} \widehat{A} \| \leq 2^{1/2} \| A A \|$ which will also be proved using the Weyl's inequality as
    \begin{equation*}
      \begin{aligned}
      \left| \| A A \| - \| \widehat{A} \widehat{A} \|\right|  \leq \| AA - \widehat{A} \widehat{A}\|. 
      \end{aligned}
    \end{equation*}
    The expected bound happens under the event $\| AA  - \widehat{A} \widehat{A}\| \leq (2^{1/2}-1)\| A A\|$. 
  \end{itemize}
  
  Finally, these computations show that
  \begin{equation} \label{eq:kappa_lasso}
    \begin{aligned}
      \P\left( \widetilde{\tau} < \frac{1}{2} \tau \right) & \leq \P\left(\left\| \E\left[ \theta(Y) \theta(Y)^T \right] - \frac{1}{n} \sum_{i=1}^n \theta(y_i) \theta(y_i)^T  \right\| \geq c \right)\\
      & \leq 2 m \exp\left(-n \frac{c^2}{2(e^2 (m+1)^2\|\theta(Y)\|_{\psi_2}^4 + e(m+1)\|\theta(Y)\|_{\psi_2}^2 c)} \right),                                                                                                                  
    \end{aligned}
  \end{equation}
  with $c = \min\left(\frac{\sqrt{2}-1}{\sqrt{2}} \frac{1}{H^2}, (\sqrt{2}-1)\|AA\| \right)$ using Lemma \ref{lem:sample_covariance_subexp}.
  The assumption is therefore satisfied with $\kappa(n) = 2 m e^{-c_2n}$ with $c_2$ defined in \eqref{eq:kappa_lasso}.

  \item[Assumption \ref{ass:holder_l}]
  First, let us identify a candidate set $\Upsilon$ for which Assumption \ref{ass:holder_l} holds. The inclusion of $\Sc$ and of $\hSc$, with high probability, will be checked in the second part of the proof.

  Let $R = \E[V^2] / \lambda$ and $\Upsilon = \{ \phi \in \Msp \, | \, \| \phi \|_2 \leq R \}$. 
  For $\phi,\psi \in \Upsilon$ and $x = (y,v), x' = (y',v') \in \Ysp \times \Vsp$ such that $x \neq x'$, one has
  \begin{equation*}
    \begin{split}
      & l(\phi,x) - l(\psi,x') - l(\phi,x) + l(\psi,x') \\
      &= \langle \phi - \psi, \theta(y)\rangle (\langle \phi + \psi, \theta(y) \rangle -2v) - \langle \psi - \phi, \theta(y')\rangle (\langle \phi + \psi, \theta(y') \rangle -2v') \\
      &\leq \|\phi - \psi\|_2 \left( \|\theta(y)\|_2 \left| \langle \phi + \psi, \theta(y) \rangle -2v \right| + \|\theta(y')\|_2 \left| \langle \phi + \psi, \theta(y') \rangle -2v' \right| \right)\\
      &\leq \|\phi - \psi\|_2 \left( \|\theta(y)\|_2 (2 R \|\theta(y)\|_2 + 2|v|) + \|\theta(y')\|_2 (2 R \|\theta(y')\|_2 + 2|v'| ) \right),
    \end{split}
  \end{equation*}
  where we successively used the identity $(a^2-b^2) = (a-b)(a+b)$ and the Cauchy-Schwartz inequality.
  The function $a$ defined by
  \begin{equation} \label{eq:a_lasso}
    a(x,x') = \left(\|\theta(y)\|_2 (2 R \|\theta(y)\|_2 + 2|v|) + \|\theta(y')\|_2 (2 R \|\theta(y')\|_2 + 2|v'|)\right)\indic_{x \neq x'},
  \end{equation}
  for $x, x' \in \Xsp$, is a pseudometric. Moreover, since $X$ and $X'$ share the same distribution, 
  \begin{equation}
    \begin{aligned}
    \|a\|_{\psi_1} &\leq 4 \left\| \|\theta(Y)\|_2 (2 R \|\theta(Y)\|_2 + 2|V|) \right\|_{\psi_1} \\
                  & \leq 4R \| \|\theta(Y)\|_2^2 \|_{\psi_1} +  4 \| \|\theta(Y)\|_2 |V| \|_{\psi_1} \\
                  & \leq 4R \| \|\theta(Y)\|_2 \|_{\psi_2}^2 + 4 \| \|\theta(Y)\|_2 \|_{\psi_2} \| |V| \|_{\psi_2} \\
                  & \leq 4R m \| \theta(Y)\|_{\psi_2}^2 + 4 \sqrt{m}\| \theta(Y)\|_{\psi_2}\|V\|_{\psi_2}.
    \end{aligned}
  \end{equation}

  Therefore, the set $\Upsilon$, $\alpha = 1$ and the function $a$ are candidates for Assumption \ref{ass:holder_l} to hold.

It is known from the proof of Assumption \ref{ass:local_holder_error} that the norm of $\phi \in \Sc$ is $\|\phi\|_1 \leq R$. Using the fact that $\|\cdot\|_2 \leq \|\cdot\|_1$, we obtain $\Sc \subset \Msp$.
Similarly, for $\phi \in \hSc$, $\|\phi\|_2 \leq \frac{1}{2n} \sum_{i=1}^n V_i^2$ so that under the event $\frac{1}{n}\sum_{i=1}^n V_i \leq 2 \E[V^2]$, the norm $\|\phi\|_2 \leq R$.
This event holds with probability at least $1 - \iota(n)$ with $\iota(n) = \eta(n)$, see \eqref{eq:concentration_Vi} in the proof of Assumption \ref{ass:convergence_bassin}.
\end{description} 
\end{proof}

\begin{proof}[{Proof of Corollary \ref{cor:lasso}}]
The proof of the first part is a direct application of Theorem \ref{thm:concentration} with the constants of Lemma \ref{lem:assumption_lasso}. The second part requires a little more work.

The total expectation  can be split accordingly to \eqref{eq:total_law_expectation} where the first term can be bounded using Lemma \ref{lem:bound_conditional_expectation}. The second term is bounded as follows.
Due to the optimality of $\phi_*$ and $\hphi$ one has
\begin{equation*}
  \begin{aligned}
  \| \phi_*\|_2 & \leq \|\phi_*\|_1 \leq \frac{1}{2\lambda} \E[V^2], \\
  \| \hphi\|_2 & \leq \|\hphi\|_1 \leq \frac{1}{2\lambda} \frac{1}{n} \sum_{i=1}^n V_i^2, \\
  \end{aligned}
\end{equation*}
see the proof of Lemma \ref{lem:assumption_lasso}. Therefore, 
\begin{equation*}
  \begin{aligned}
    \E\left[ \vartheta(\hSc ; \Sc)^2 \indic_{\bX \notin \Ac}\right] & \leq \frac{1}{2}\E[V^2]^2\lambda^{-2}p_n + \frac{1}{2} \E[V^4 \indic_{\bX \notin \Ac}] \\ 
    & \leq \frac{1}{2}\E[V^2]^2\lambda^{-2}p_n + \frac{1}{2} \E[V^8]^{\frac{1}{2}} \sqrt{p_n} \\ 
    & \leq (2\|V\|^2_{\psi_2}+ 2^5\|V\|^4_{\psi_2}) \lambda^{-2} \sqrt{p_n}.
  \end{aligned}
\end{equation*}
\end{proof}

\subsection{Entropic-Wasserstein Barycenters} \label{sec:proof_entropic_wass}

\begin{proof}[Proof of Lemma \ref{lem:assumption_regbar}]
The proof consists in checking the assumptions one by one.
\begin{description}
  \item[Assumption \ref{ass:existstence_minimizer}] 
  Since $\Omega$ is bounded there exists $M >0$ such that $\int_{\Pc_2(\Omega)} \E_{X \sim \psi}[X^2] d\mu(\psi) \leq M^2$ for all measures $\mu \in \Pc_2(\Msp)$. Using Proposition 2.1 of \citet{carlier2021entropic} allows to conclude that $\phi_*$ and $\hphi$ exist and are unique.

  \item[Assumption \ref{ass:local_holder_error}]
  The function $W_2^2(\cdot,\psi)$ is convex from Proposition 7.17 of \citet{santambrogio2015optimal} for all $\psi \in \Pc_2(\Omega)$ for the linear structure on $\Pc_2(\Omega)$. This means that
  \begin{equation*}
    W_2^2(\phi_t,\psi) \leq (1-t) W_2^2(\phi_0,\psi) + t W_2^2(\phi_1,\psi),
  \end{equation*}
   for any pair of measures $\phi_0, \phi_1 \in \Pc_2(\Omega)$ and for any $t \in [0,1]$ where $\phi_t = (1-t)\phi_0 + t \phi_1$.
   Integrating with respect to $\mu$ or $\mu_n$ shows that the functions $\ff$ and $\hff$ are convex.

  On the other hand, some computations show that
  \begin{equation*}
    \RR(\phi) = \RR(\psi) + \langle \nabla \RR(\psi), \phi - \psi\rangle + \textrm{KL}(\phi|\psi),
  \end{equation*}
  with $\nabla \RR(\psi) = \log(\psi)$ and $\textrm{KL}$ being the Kullback-Liebler divergence between $\phi$ and $\psi$ defined by
  \begin{equation*}
    \textrm{KL}(\phi|\psi) = \int_{\Omega} \log\left(\frac{\phi(z)}{\psi(z)} \right) \phi(z) dz.
  \end{equation*}
  The Pinsker's inequality states that $\vartheta(\phi,\psi) \leq \sqrt{\frac{1}{2} \textrm{KL}(\phi|\psi)}$. This means that $\RR$ is $2$-strongly convex with respect to the total variation distance for the linear structure.

  Hence $\JJ$ and $\hJJ$ are both $2$-strongly convex with respect to the total variation distance so that
  \begin{equation*}
    2 \vartheta(\phi,\psi)^2 \leq \JJ(\phi) - \JJ_*,
  \end{equation*}
  for all $\phi \in \Pc_2^{ac}(\Omega)$. The similar inequality holds for $\hJJ$.

  \item[Assumption \ref{ass:convergence_bassin}] Since Assumption \eqref{ass:local_holder_error} is verified with $\JJ_0 = +\infty$, this assumption is verified with $\eta = 0$.

  \item[Assumption \ref{ass:lojasewicz_modulus}] Since $\hJJ$ is $2$-strongly convex with respect to the total variation distance, it holds that $\htau = \tau$ with probability $1$. This implies that the assumption is verified with $\kappa = 0$.

  \item[Assumption \ref{ass:holder_l}]
  Prior to proving this assumption, the Kantorovich duality is briefly introduced \citep[see][for more details]{villani2003topics}.
  Given two measures $\phi$ and $\psi$ and two maps $g \in L^1(\phi)$ and $h \in L^1(\psi)$ satisfying $g(y) + h(z) \geq \langle y, z\rangle$ for $\phi$-a.e.\ $y$ and $\psi$-a.e.\ $z$, then by definition of $W_2$
  \begin{equation*}
    \frac{1}{2} W_2^2(\phi,\psi) \geq \int_\Omega \left(\frac{\|\cdot\|^2}{2} - g\right) d\phi + \int_\Omega \left(\frac{\|\cdot\|^2}{2} - h \right)d\psi.
  \end{equation*}
  Due the Kantorovich duality, it is known that the equality holds for some pair of $(g,h) = (\zeta_{\phi \to \psi}, \zeta_{\phi \to \psi}^*)$ with $\zeta_{\phi \to \psi}$ being a lower semi-continuous proper convex function and where $\zeta_{\phi \to \psi}^* = \zeta_{\psi \to \phi}$ is the convex conjugate of $\zeta_{\phi \to \psi}$. The map $\zeta_{\phi \to \psi}$ is called the Kantorovich potential of the pair of measures $(\phi,\psi)$ and therefore satisfies
\begin{equation*}
    \frac{1}{2} W_2^2(\phi,\psi) = \int_\Omega \left(\frac{\|\cdot\|^2}{2} - \zeta_{\phi \to \psi}\right) d\phi + \int_\Omega \left(\frac{\|\cdot\|^2}{2} - \zeta_{\psi \to \phi} \right)d\psi.
  \end{equation*}

  Back to the proof of the assumption, let $\phi_0, \phi_1 \in \Msp$ and $\psi_0,\psi_1 \in \Xsp$.
  By Kantorovich duality \citep{villani2003topics},
  \begin{equation*}
    \begin{aligned}
      \frac{1}{2} W_2^2(\phi_0,\psi_0) &= \int_\Omega \left(\frac{\|\cdot\|^2}{2} - \zeta_{\psi_0 \to \phi_0} \right) d\psi_0 + \int_\Omega \left(\frac{\|\cdot\|^2}{2} - \zeta_{\phi_0 \to \psi_0} \right) d\phi_0, \\
      \frac{1}{2} W_2^2(\phi_1,\psi_0) & \geq \int_\Omega \left( \frac{\|\cdot\|^2}{2} - \zeta_{\psi_0 \to \phi_0} \right) d\psi_0 + \int_\Omega \left( \frac{\|\cdot\|^2}{2} - \zeta_{\phi_0 \to \psi_0} \right) d\phi_1,
    \end{aligned}
  \end{equation*}
  so that 
  \begin{equation*}
    \frac{1}{2} W_2^2(\phi_0,\psi_0) - \frac{1}{2} W_2^2(\phi_1,\psi_0) \leq \int_\Omega \left( \frac{\|\cdot\|^2}{2} - \zeta_{\phi_0 \to \psi_0} \right) d(\phi_0 - \phi_1).
  \end{equation*}
  Therefore,
  \begin{equation*}
    \begin{aligned}
    & \frac{1}{2}W_2^2(\phi_0,\psi_0) - \frac{1}{2}W_2^2(\phi_1,\psi_0) + \frac{1}{2} W_2^2(\phi_1,\psi_1) - \frac{1}{2} W_2^2(\phi_0,\psi_1) \\
    & \quad \quad \leq \int_\Omega (\zeta_{\phi_1 \to \psi_1} - \zeta_{\phi_0 \to \psi_0})  d(\phi_0 - \phi_1).
    \end{aligned}
  \end{equation*}

  Assuming without loss of generality that the point $0_{\R^d} \in \Omega$, the Kantorovich $\zeta$ are $2\diam(\Omega)$-Lipschitz so that 
\begin{equation*}
  \begin{split}
    \frac{1}{2} W_2^2(\phi_0,\psi_0) - \frac{1}{2} W_2^2(\phi_1,\psi_0) + \frac{1}{2} W_2^2(\phi_1,\psi_1) - \frac{1}{2} W_2^2(\phi_0,\psi_1)  & \leq 4 \diam(\Omega) W_1(\phi_0,\phi_1) \\
    & \leq 4 \diam(\Omega)^2 \vartheta(\phi_0,\phi_1).
    \end{split}
  \end{equation*}
  where the last inequality is obtained through the transport inequality $W_1(\phi_0,\phi_1) \leq \diam(\Omega) \vartheta(\phi_0, \phi_1)$ \citep[see][Theorem 4]{gibbs2002choosing}.
  Therefore, the assumption is verified with $\Upsilon = \Msp$, $a = 4 \diam(\Omega)^2$, $\|a\|_{\psi_1} = 4 \diam(\Omega)^2$ and $\iota = 0$.
\end{description}
\end{proof}

\bibliography{biblio}

\end{document}